\documentclass[10pt]{amsart}
\usepackage[usenames,dvipsnames]{xcolor} %for coloured links
\usepackage[pdftitle={},
pdfauthor={Jacopo Borga, Raul Penaguiao},
colorlinks=true,linkcolor=JungleGreen,urlcolor=OliveGreen,citecolor=PineGreen,bookmarks=true,bookmarksopen=true,bookmarksopenlevel=2,unicode=true,linktocpage]{hyperref}

\usepackage[english]{babel}%for the language
\usepackage{amsmath}  %for several maths commands
\usepackage{graphicx}  % for figures, here an example:
\usepackage{amsthm} % for different theorems styles		
\usepackage[capitalize]{cleveref} %for fast citations (Keep it after \usepackage{amsthm})				
\usepackage{amsmath,bm}  % For bold notation
\usepackage{amsfonts} %for several math tools
\usepackage{mathtools} %for several math tools like\coloneqq
\usepackage{dsfont}% for fancy sign like \mathds{1}
\usepackage{bbold}
\usepackage{caption}

\newtheorem{thm}{Theorem}[section]

\newtheorem{prop}[thm]{Proposition}
\newtheorem{lem}[thm]{Lemma}
\newtheorem{lm}[thm]{Lemma}
\newtheorem{conj}[thm]{Conjecture}

\theoremstyle{definition}
\newtheorem{defn}[thm]{Definition}
\newtheorem{defin}[thm]{Definition}

\newtheorem{exmp}[thm]{Example}

\newtheorem{fact}[thm]{Fact}
\newtheorem{obs}[thm]{Observation}

\crefname{Lemma}{Lemma}{Lemmas}
\crefname{Proposition}{Proposition}{Propositions}

\theoremstyle{remark}
\newtheorem{rem}[thm]{Remark}

\usepackage{todonotes}

\def\monodown{\nu}
\def\Z{\mathbb{Z}}

\def\pcoc{\widetilde{\coc}}
\def\SS{\mathcal{S}} 
\def\CCC{\mathcal{C}} 
\def\poc{\widetilde{\oc}}
\def\ValGraph[#1]{\mathcal{O}v(#1)}
\def\AValGraph[#1,#2]{\mathcal{O}v^{#2}(#1)}
\def\OvMon[#1,#2]{\mathfrak{C}\mathcal{O}v^{#2}(#1)}
\def\Wpath[#1, #2]{\mathfrak{C}W^{#2}_{#1}}

\DeclareMathOperator{\coc}{c-occ}
\DeclareMathOperator{\pat}{pat}
\DeclareMathOperator{\Av}{Av}
\DeclareMathOperator{\oc}{occ}
\DeclareMathOperator{\conv}{conv}
\DeclareMathOperator{\st}{s}

\DeclareMathOperator{\ar}{a}
\DeclareMathOperator{\lb}{lb}
\DeclareMathOperator{\be}{beg}
\DeclareMathOperator{\en}{end}
\DeclareMathOperator{\std}{std} 
\DeclareMathOperator{\im}{im}

\DeclareMathOperator{\per}{per} 
\DeclareMathOperator{\dist}{dist}


\begin{document}
\title[The feasible regions for consecutive patterns]{The feasible regions for consecutive patterns of pattern-avoiding permutations}	

\author{Jacopo Borga}
\address{Department of Mathematics, Stanford University}
\email{jborga@stanford.edu}
\urladdr{https://www.jacopoborga.com/} % Delete if not wanted.

\author{Raul Penaguiao}
\address{Department of Mathematics, San Francisco State University}
\email{raulpenaguiao@sfsu.edu}
\urladdr{https://raulpenaguiao.github.io/} % Delete if not wanted.

\maketitle

\begin{abstract}
We study the feasible region for consecutive patterns of pattern-avoiding permutations. More precisely, given a family $\mathcal C$ of permutations avoiding a fixed set of patterns, we consider the limit of proportions of consecutive patterns on large permutations of $\mathcal C$. These limits form a region, which we call the \emph{consecutive patterns feasible region for $\mathcal C$}. 

We determine the dimension of the consecutive patterns feasible region for all families $\mathcal C$ closed either for the direct sum or the skew sum. These families include for instance the ones avoiding a single pattern and all substitution-closed classes. We further show that these regions are always convex and we conjecture that they are always polytopes.
We prove this conjecture when $\mathcal C$ is the family of $\tau$-avoiding permutations, with either $\tau$ of size three or $\tau$ a monotone pattern. 
Furthermore, in these cases we give a full description of the vertices of these polytopes via cycle polytopes.

Along the way, we discuss connections of this work with the problem of packing patterns in pattern-avoiding permutations and to the study of local limits for pattern-avoiding permutations.
\end{abstract}

\section{Introduction}

Pattern-avoiding permutations are well-known to show very different behaviour according to the set of patterns they avoid. This makes it extremely difficult to obtain results that are valid for every family of pattern-avoiding permutations. This belief is also confirmed by the available literature, where most of the results are restricted to some specific families of pattern-avoiding permutations. There is one famous exception: Marcus and Tardos~\cite{MR2063960} in 2004 proved the Stanley-Wilf conjecture. Formulated independently by Richard Stanley and Herbert Wilf, it states that for every permutation $\tau$, there is a constant $C$ depending on $\tau$ such that the number of permutations of length $n$ which avoid $\tau$ is at most $C^n$. 

In this paper we introduce the \emph{consecutive patterns feasible region} for a family $\mathcal C$ of pattern-avoiding permutations. Several motivations for studying these regions are provided in \cref{sect:motiv}.
We prove a general result --  computing their dimension -- that holds for instance for all families $\mathcal C$ avoiding a fixed pattern (see \cref{thm:dim_is_tight}).
We also study in depth the cases when $\mathcal C$ is the family of $\tau$-avoiding permutations, with $\tau$ of size three or $\tau$ a monotone pattern. For these families, we are able to give a complete description of the regions as polytopes (see Theorems \ref{thm:312-avoiding} and \ref{thm:feasregmonotones*}).

\subsection{The consecutive patterns feasible regions}\label{sect:motiv}

The study of limits of (random) pattern-avoiding permutations is a very active field in combinatorics and discrete probability theory. There are two main ways of investigating these limits:

\begin{itemize}
	\item The most classical one is to look at the limits of various statistics for pattern-avoiding permutations. For instance, the limiting distributions of the longest increasing subsequences in uniform pattern-avoiding permutations have been considered in  \cite{deutsch2003longest,mansour2019longest,bassino2021linear}. Another example is the general problem of studying the limiting distribution of the number of occurrences of a fixed pattern $\pi$ in a uniform random permutation avoiding a fixed set of patterns when the size tends to infinity (see for instance Janson's papers~\cite{janson2017patterns,janson2017patterns321,janson2018patterns}, where the author studied this problem in the model of uniform permutations avoiding a fixed family of patterns of size three). Many other statistics have been considered: for instance in \cite{bukata2018distributions} the authors studied the distribution of ascents, descents, peaks, valleys, double ascents and double descents over pattern-avoiding permutations. 
	
	\item The second way  is to look at geometric limits of large pattern-avoiding permutations. Two main notions of convergence for permutations have been defined: a global notion of convergence (called permuton convergence, \cite{MR2995721}) and a local notion of convergence\footnote{A third and new notion of convergence was recently introduced in \cite{bevan2020independence}; it interpolates between the two main notions mentioned in the paper.} (called Benjamini--Schramm convergence, \cite{borga2018local}). For an intuitive explanation of them we refer the reader to \cite[Section 1.1]{borga2019feasible}, where additional references can be found. We just mention here that permuton convergence is equivalent to the convergence of all \textit{pattern density} statistics (see \cite[Theorem 2.5]{bassino2017universal}); and Benjamini--Schramm convergence is equivalent to the convergence of all \textit{consecutive pattern density} statistics (see \cite[Theorem 2.19]{borga2018local}).
	The latter is the subject of this paper.
\end{itemize}

In this paper we study the feasible region for consecutive patterns of pattern-avoiding permutations.
This object is strongly connected with both the statistical study and the geometric study of permutations presented above (explanations are given below). We start by defining this region and presenting our main results. %Then we comment on these connections.

\medskip

Let $\mathcal{S}_k$ denote the set of permutations of size $k$ and $\SS$ the set of all permutations. Many basic concepts on permutations will be recalled in \cref{sect:notation}. In this introduction, we use the classical terminology and we briefly introduce essential notation along the way, like $\pcoc(\pi,\sigma)$ which denotes the proportion of  consecutive occurrences of a pattern $\pi$ in a permutation $\sigma$. Given a set of patterns $B\subset\mathcal{S},$ we denote by $\Av_n(B)$ the set of $B$-avoiding permutations of size $n$, and by $\Av(B)\coloneqq\bigcup_{n\in\Z_{\geq 0}}\Av_n(B)$ the set of $B$-avoiding permutations of arbitrary finite size. We denote by $|\Av_n(B)|$ the cardinality of $\Av_n(B)$. 

We consider the \emph{consecutive patterns feasible region for $\Av(B)$}, defined by
\begin{multline*}
P^{B}_k \coloneqq \{\vec{v}\in [0,1]^{\SS_k} \big| \exists (\sigma^m)_{m\in\Z_{\geq 1}} \in \Av(B)^{\Z_{\geq 1}} \text{ such that }\\
|\sigma^m| \to \infty \text{ and }  \pcoc(\pi, \sigma^m ) \to \vec{v}_{\pi}, \forall \pi\in\SS_k \}.
\end{multline*}
%For different choices of the set $B$ of avoided-patterns, we refer to these regions as \emph{pattern-avoiding feasible regions}. 
In words, the region $P^{B}_k$ is formed by the $k!$-dimensional vectors $\vec{v}$ for which there exists a sequence of permutations in $\Av(B)$ whose size tends to infinity and whose proportion of consecutive patterns of size $k$ tends to $\vec{v}$. For simplicity, whenever $B=\{\tau\}$ we simply write $P^{\tau}_k$ for $P^{\{\tau\}}_k$ (and we use the same convention for related notation).

Our first main result is the following one. We denote by $\oplus$ the \emph{direct sum} of two permutations and by $\ominus$ the \emph{skew sum} (definitions are in \cref{sect:notation}). 

\begin{thm}\label{thm:dim_is_tight}
	Fix $k\in\Z_{\geq 1}$ and a set of patterns $B\subset\mathcal{S}$ such that the family $\Av(B)$ is closed either for the $\oplus$ operation or $\ominus$ operation. The feasible region $P^{B}_k$ is closed and convex. Moreover,
	$$\dim(P^{B}_k)= |\Av_k(B)|-|\Av_{k-1}(B)|.$$
\end{thm}

%In particular, $P^{\Av ( \tau )}_k$ is always convex for any pattern $\tau\in\SS$, as $\Av ( \tau )$ is either closed for the $\oplus$ operation (whenever $\tau$ is $\oplus$-indecomposable) or closed for the $\ominus$ operation (whenever $\tau$ is $\ominus$-indecomposable).
\begin{rem}
	We emphasize that the hypothesis in \cref{thm:dim_is_tight} is not superfluous.
	Indeed, for some sets of patterns $B$, the region $P^{B}_k$ is not convex.
	For instance, if $B= \{132, 213, 231, 312\}$, then $\Av(B)$ is the set of monotone permutations. Therefore, the resulting feasible region for consecutive patterns is formed by two distinct points, hence it is not convex. 
\end{rem}

\begin{rem}
	For any fixed pattern $\tau \in \mathcal S$, the family $\Av ( \tau )$ is either closed for the $\oplus$ operation (whenever $\tau$ is $\oplus$-indecomposable) or closed for the $\ominus$ operation (whenever $\tau$ is $\ominus$-indecomposable).
	
	Therefore, by \cref{thm:dim_is_tight}, for every pattern $\tau\in\mathcal S$, the region $P^{\tau}_k$ is closed and convex, and
	$$\dim(P^{\tau}_k)= |\Av_k(\tau)|-|\Av_{k-1}(\tau)|.$$
\end{rem}

\begin{rem}
	Our theorem is also valid for various families of pattern-avoiding permutations avoiding multiple patterns. For instance all substitution closed-classes satisfy the hypothesis of \cref{thm:dim_is_tight}. Substitution closed-classes were first studied by Albert and Atkinson~\cite{MR2170110} and received much attention in various consecutive works. We refer to \cite[Section 2.2]{MR4115736} for an introduction to substitution closed-classes. We also remark that Benjamini--Schramm limits of substitution closed-classes were recently investigated in \cite{MR4115736}.
	
	A well-known example of a substitution closed-class is given by \emph{separable permutations}. These form the family of permutations avoiding the patterns $2413$ and $3142$ and they have been consider in several mathematical fields (in enumerative combinatorics and algorithmics \cite{MR1620935,MR3359905}, in real analysis \cite{MR3702027}, and in probability theory  \cite{MR1093199,MR3813988}). 
\end{rem}

Our second main result shows that the consecutive patterns feasible regions $P_k^{\tau}$ for $\tau$ of size three or $\tau$ a monotone pattern is a polytope, and gives a description of the corresponding vertices. Precise statements are given in \cref{thm:312-avoiding} and \cref{thm:feasregmonotones*}, after having introduced the required notation.
We finally conjecture (see \cref{conj:dim_is_tight}) that, whenever $\mathcal C $ is closed for the operation $\oplus $ or $\ominus $, the feasible region is a polytope.

\medskip

We now comment on the connection between the consecutive patterns feasible regions and the two ways of studying limits of pattern-avoiding permutations mentioned before. For the first one, i.e.\ the study of various statistics for pattern-avoiding permutations, the statistic that we consider here is the number of consecutive occurrences of a pattern (see, for instance, the survey of Elizalde \cite{MR3526425} for various motivations for studying these patterns). For the second one, i.e.\ the study of geometric limits, the relation is with Benjamini--Schramm limits (investigated, for instance, in \cite{bevan2019permutations,borga2018local,MR4115736}). In particular, having a precise description of the regions $P^{B}_k$ for all $k\in\Z_{\geq 1}$  determines all the Benjamini--Schramm limits that can be obtained through sequences of permutations in $\Av(B)$. 

\medskip

An orthogonal motivation for investigating the pattern avoiding feasible regions is the problem of \emph{packing patterns} in pattern avoiding permutations. The classical question of \emph{packing patterns} in permutations consists in describing the maximum number of occurrences of a pattern $\pi$ in any permutation of $\SS_n$ (see for instance \cite{MR2695616,MR1887086,MR2114184}). More recently, a similar question in the context of pattern-avoiding permutations has been addressed by Pudwell~\cite{Pudwell_pak_av}. It consists in describing the maximum number of occurrences of a pattern $\pi$ in any pattern-avoiding permutation.
Describing the feasible region for consecutive patterns of pattern-avoiding permutations $P^{B}_k$ is a fundamental step for solving the question of finding the asymptotic maximum number of consecutive occurrences of a pattern $\pi\in \SS_k$ in large permutations of $\Av(B)$ (indeed the latter problem can be translated into a linear optimization problem in the feasible region $P^{B}_k$).

\medskip

Additional motivations for studying the regions $P^{B}_k$ are the novelties of the results in this paper compared with a previous work \cite{borga2019feasible}. There, the consecutive patterns feasible region for the set of \emph{all} permutations $\mathcal S$ was introduced as: 
\begin{multline*}
P_k \coloneqq \big\{\vec{v}\in [0,1]^{\SS_k} \big| \exists (\sigma^m)_{m\in\Z_{\geq 1}} \in \SS^{\Z_{\geq 1}} \text{such that}\\
|\sigma^m| \to \infty \text{ and }  \pcoc(\pi, \sigma^m ) \to \vec{v}_{\pi}, \forall \pi\in\SS_k \big\}\, ,
\end{multline*}
and studied, specifically giving its dimension, establishing that it is a polytope, and describing all its vertices and facets. %For a more precise discussion on this point see \cref{sect:prev_res,sect:novelties}.
We refer the reader to \cite[Section 1.1]{borga2019feasible} for motivations to investigate this region and to \cite[Section 1.2]{borga2019feasible} for a summary of the related literature. 

\begin{rem}
	We recall that \emph{(classical) patterns feasible regions} were also considered in the literature. We refer to \cite[Section 1.2]{borga2019feasible} for a complete review of the related literature. We also remark that determining the dimension of the (classical) patterns feasible regions for the set of \emph{all} permutations is still an open problem (see \cite[Conjecture 1.3]{borga2019feasible}). 
\end{rem}

\subsection{Previous results on the standard feasible region for consecutive patterns}\label{sect:prev_res}

Before presenting our additional results on the consecutive patterns feasible regions, we recall two key definitions from \cite{borga2019feasible} and review some results presented in that paper.

\begin{defn}
	\label{defn:ov_graph}
	The \emph{overlap graph} $\ValGraph[k]$ is a directed multigraph with labelled edges, where the vertices are elements of $\SS_{k-1}$ and for every $\pi\in\SS_{k}$ there is an edge labelled by $\pi$ from the pattern induced by the first $k-1$ indices of $\pi$ to the pattern induced by the last $k-1$ indices of $\pi$.
\end{defn}
For an example with $k=3$ see the left-hand side of \cref{fig:P_3_and_rest2}.

\begin{defn}
	Let $G=(V,E)$ be a directed multigraph.
	For each non-empty cycle $\mathcal{C}$ in $G$, define $\vec{e}_{\mathcal{C}}\in \mathbb{R}^{E}$ such that
	$$(\vec{e}_{\mathcal{C}})_e \coloneqq \frac{\text{\# of occurrences of $e$ in $\mathcal{C}$}}{|\mathcal{C}|}, \quad \text{for all}\quad e\in E. $$
	We define the \emph{cycle polytope} of $G$ to be the polytope 
	$$P(G) \coloneqq \conv \{\vec{e}_{\mathcal{C}} | \, \mathcal{C} \text{ is a simple cycle of } G \}.$$
\end{defn}

We recall some results from \cite{borga2019feasible}. We start with the following consequence of \cite[Proposition 6]{gleiss2001circuit}.

\begin{prop}[Proposition 1.7 in \cite{borga2019feasible}]\label{thm:dim_cycle_pol}
	The cycle polytope of a strongly connected directed multigraph $G=(V,E)$ has dimension $|E|-|V|$.
\end{prop}

Our main result in \cite{borga2019feasible} is the following one.

\begin{thm}[Theorem 1.6 in \cite{borga2019feasible}] \label{thm:main_res_prev_art}
	$P_k$ is the cycle polytope of the overlap graph $\ValGraph[k]$. Its dimension is $k! - (k-1)!$ and its vertices are given by the simple cycles of $\ValGraph[k]$.
\end{thm}

An instance of the result above is depicted in \cref{fig:P_3_and_rest2}.

\begin{figure}[htbp]
	\centering
	\includegraphics[scale=0.9]{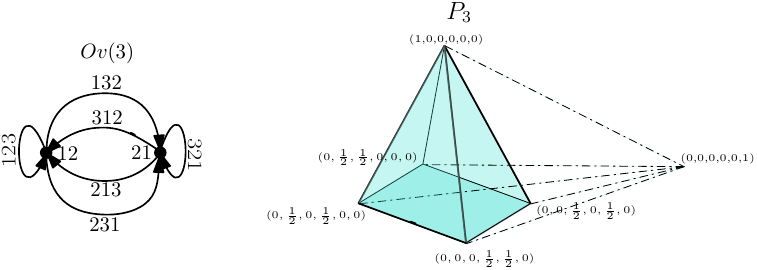}
	\captionsetup{width=\textwidth}
	\caption{The overlap graph $\ValGraph[3]$ and the four-dimensional polytope $P_3$. The coordinates of the vertices correspond to the patterns $(123,231,312,213,132,321)$ respectively. Note that the top vertex (resp.\ the right-most vertex) of the polytope corresponds to the loop indexed by $123$ (resp.\ $321$); the other four vertices correspond to the four cycles of length two in $\ValGraph[3]$. 
	We highlight in light-blue one of the six three-dimensional faces of $P_3$. This face is a pyramid with a square base. The polytope itself is a four-dimensional pyramid, whose base is the blue face. \cref{thm:main_res_prev_art} implies that $P_3$ is the cycle polytope of $\ValGraph[3]$. \label{fig:P_3_and_rest2}}
\end{figure}

We also recall for later purposes the following construction related to the overlap graph $\ValGraph[k]$.
Given a permutation $\sigma\in\SS_m$, for some $m\geq k$, we can associate with it a walk $W_k(\sigma)=(e_1,\dots, e_{m-k+1})$ in $\ValGraph[k]$ of size $m-k+1$, where $e_i$ is the edge of $\ValGraph[k]$ labelled by the pattern of $\sigma$ induced by the indices from $i$ to $i+k-1$.
The map $W_k$ is not injective, but in \cite{borga2019feasible} we proved the following. 

\begin{lem}[Lemma 3.8 in \cite{borga2019feasible}]\label{lemma:pathperm}
	Fix $k\in\Z_{\geq 1}$ and $m\geq k$. The map $W_k$, from the set $\SS_{m}$ of permutations of size $m$ to the set of walks in $\ValGraph[k]$ of size $m-k+1$, is surjective.	
\end{lem}
This lemma was a key step in the proof of \cref{thm:main_res_prev_art}.

\subsection{Additional results on the consecutive patterns feasible regions}\label{sect:novelties}

We start with a natural generalization of \cref{defn:ov_graph} to pattern-avoiding permutations.

\begin{defn}
	\label{defn:ov_graph_av_perm}
	Fix a set of patterns $B\subset\mathcal{S}$ and $k\in\Z_{\geq 1}$. The \emph{overlap graph} $\AValGraph[k,B]$ is a directed multigraph with labelled edges, where the vertices are elements of $\Av_{k-1}(B)$ and for every $\pi\in\Av_{k}(B)$ there is an edge labelled by $\pi$ from the pattern induced by the first $k-1$ indices of $\pi$ to the pattern induced by the last $k-1$ indices of $\pi$. 
\end{defn}

Informally, $\AValGraph[k,B]$ arises simply as the restriction of $\AValGraph[k,]$ to all the edges and vertices in $\Av(B)$. We have the following result, which is proved in \cref{sec:topopro}.

%We denote by $\oplus$ the \emph{direct sum} of two permutations and by $\ominus$ the \emph{skew sum} (recall that precise definitions are given in \cref{sect:notation}). 

\begin{prop}\label{prop:inclusions}
	Fix $k\in\Z_{\geq 1}$. For all sets of patterns $B\subset\mathcal{S}$, the feasible region $P^{B}_k$ satisfies $P^{B}_k\subseteq P(\AValGraph[k,B]) \subseteq P_k $. 	
%	Moreover, if $\Av(B)$ is closed either for the $\oplus$ operation or $\ominus$ operation then the feasible region $P^{B}_k$ is convex and
	%$\dim(P^{B}_k)\leq |\Av_k(B)|-|\Av_{k-1}(B)|$.
\end{prop}

%In particular, $P^{\Av ( \tau )}_k$ is always convex for any pattern $\tau\in\SS$, as $\Av ( \tau )$ is either closed for the $\oplus$ operation (whenever $\tau$ is $\oplus$-indecomposable) or closed for the $\ominus$ operation (whenever $\tau$ is $\ominus$-indecomposable).

%For some sets of patterns $B$, the region $P^{B}_k$ is not even convex.
%For instance, if $B= \{132, 213, 231, 312\}$, then $\Av(B)$ is the set of monotone permutations. Therefore, the resulting consecutive patterns feasible region is formed by two distinct points, hence it is not convex. 
%This shows that the last hypothesis in \cref{prop:inclusions} is not superfluous.

We will show later in \cref{fact:not_always_equal} that sometimes $P^{B}_k\neq P(\AValGraph[k,B])$ even if $B=\{\tau\}$ (see also the bottom part of \cref{fig:P_3_and_rest}). Note that this makes the proof of \cref{thm:dim_is_tight} less straightforward. Indeed, only the upper bound $\dim(P^{B}_k)\leq |\Av_k(B)|-|\Av_{k-1}(B)|$ can be deduced from \cref{prop:inclusions} together with \cref{thm:dim_cycle_pol}. As we will see in \cref{sec:topopro}, for the complete proof of \cref{thm:dim_is_tight} we use a new approach.

\medskip

%but we believe that the bound on the dimension of the feasible regions, given in \cref{prop:inclusions}, is tight whenever $|B|=1$.
%Indeed, we wish to show that any feasible region of this form contains a contraction of $P(\AValGraph[k,B])$, from which the following conjecture would follow:
\cref{thm:dim_is_tight} states that the regions $P^{B}_k$ are convex for every choice of $B$ such that $\Av(B)$ is closed either for the $\oplus$ operation or $\ominus$ operation.  We further believe that the following stronger result holds.

\begin{conj}\label{conj:dim_is_tight}
	Fix $k\in\Z_{\geq 1}$ and a sets of patterns $B\subset\mathcal{S}$ such that the family $\Av(B)$ is closed either for the $\oplus$ operation or $\ominus$ operation.
	The feasible region $P^{B}_k$ is a polytope.
\end{conj}

%It is natural to wonder what happens for $|B|\geq 2$.
%In the case $B=\{132, 213, 231, 312\}$ described above, the feasible region is not even convex.
%However, using the general notion of dimension in real spaces due to Hausdorff, we can still talk about the dimension of this reagion, and we have that $\dim P^{B}_k = 0$, which seems to agree with the prediction of our conjecture for $k\geq 3$.
%Despite that, it illustrates who much wilder the case $|B|\geq 2$ may get, and we prefer to state \cref{conj:dim_is_tight} only in the case $|B|=1$.
%However, when we consider $B=\mathcal{S}_{k+1}$, then $P^{B}_k = \emptyset $, which by convention has dimension $-1$, whereas we have $|\Av_k(B)|-|\Av_{k-1}(B)|= k! - (k-1)!$.

We will prove that \cref{conj:dim_is_tight} is true when $|\tau |= 3$ or when $\tau$ is a monotone pattern, i.e.\ $\tau = n \cdots 1$ or $\tau = 1 \cdots n$, for $n\in\Z_{\geq 2}$. %Furthermore, we will completely describe the feasible regions $P^{\tau}_k$ for such patterns $\tau$.
By symmetry, we only need to study the cases $\tau = 312$ and $\tau = n \cdots 1$ for $n\in\Z_{\geq 2}$. Indeed, every other permutation arises as compositions of the reverse map (symmetry of the diagram w.r.t.\ the vertical axis) and the complementation map (symmetry of the diagram w.r.t.\ the horizontal axis) of the permutations $\tau = 312$ or $\tau = n \cdots 1$ for $n\in\Z_{\geq 2}$. Beware that the inverse map (symmetry of the diagram w.r.t.\ the principal diagonal) cannot be used since it does not preserve consecutive pattern occurrences. 

\medskip

We conclude this introduction by describing precisely the polytopes $P^{312}_k$ and $P^{\monodown_n}_k$ for all $\monodown_n=n\dots1$. 

\bigskip

When $\tau = 312$ the description of the region $P^{312}_k$ is quite simple; indeed we have the following result.

\begin{thm}
\label{thm:312-avoiding}
Fix $k\in\Z_{\geq 1}$. The feasible region $P^{312}_k$ is the cycle polytope of the overlap graph $\AValGraph[k,312]$. 
%Its dimension is $C_k - C_{k-1}$, where $C_k$ is the $k$-th Catalan number, and its vertices are given by the simple cycles of $\AValGraph[k,312]$.
\end{thm}

An instance of the result above is depicted on top of \cref{fig:P_3_and_rest}.

\begin{figure}[htbp]
	\centering
	\includegraphics[scale=0.9]{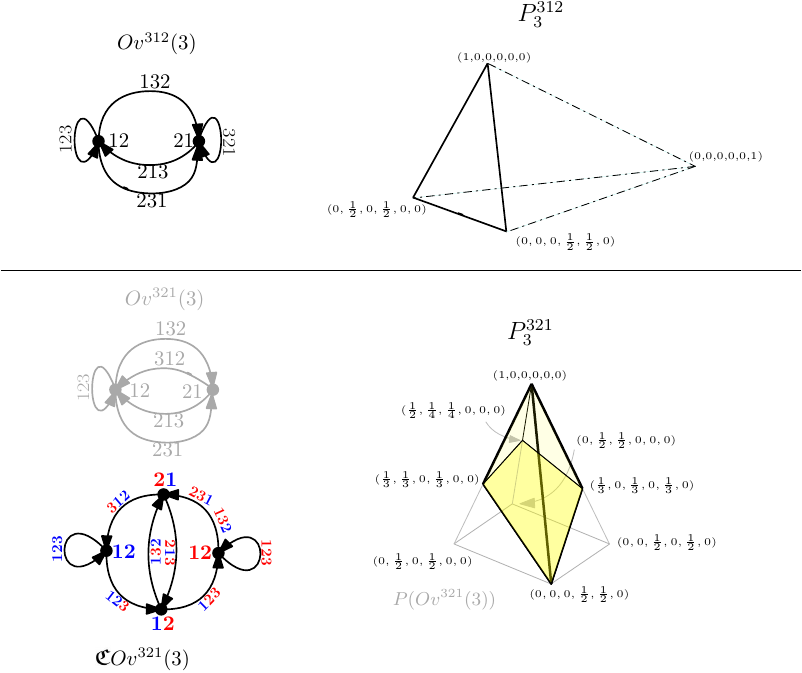}
	\captionsetup{width=\textwidth}
	\caption{We use the same conventions as in \cref{fig:P_3_and_rest2} for the coordinates of the vertices of the polytopes. \textbf{Top:} The overlap graph $\AValGraph[3,312]$ and the  three-dimensional polytope $P^{312}_3$. Note that $P^{312}_3\subset P_3$ (recall \cref{fig:P_3_and_rest2}). From \cref{thm:312-avoiding} we have that $P^{312}_3$ is the cycle polytope of $\AValGraph[3,312]$. \textbf{Bottom:} In light grey the overlap graph $\AValGraph[3,321]$ and the corresponding three-dimensional cycle polytope $P(\AValGraph[3,321])$, which is strictly larger than $P^{321}_3$. The latter feasible region is highlighted in yellow. 
	From \cref{thm:feasregmonotones*} we have that $P^{321}_k$ is the projection (defined precisely in \cref{thm:feasregmonotones}) of the cycle polytope of the coloured overlap graph $\OvMon[3, 321]$ (see \cref{defn:ov_graph_av_monotoneperm} for a precise description). This graph is plotted in the bottom-left side. Note that $P^{321}_3\subset P_3$. \label{fig:P_3_and_rest}}
\end{figure}

\bigskip 

Despite the description of the region $P^{312}_k$ is quite simple, for some patterns $\tau$, the precise description of the region $P^\tau_k$ is quite involved, as we will see in the next result. 
We fix $\monodown_n = n \cdots 1 $ for $n\in\Z_{\geq 2}$, the decreasing pattern of size $n$, and an integer $k\in \Z_{\geq 1}$.

We start with the following fact (compare this with the bottom part of \cref{fig:P_3_and_rest}), which shows that the study of the monotone case deviates significantly from the one in \cref{thm:312-avoiding}.

\begin{fact}\label{fact:not_always_equal}
 The cycle polytope $P(\AValGraph[3,321])$ is different from the feasible region $P^{321}_k$.
\end{fact}

\begin{proof}
	Consider the vector $\vec{v}=(0,1/2,1/2,0,0,0)$, where the coordinates of the vector correspond to the patterns $(123,231,312,213,132,321)$. We show that $\vec{v}\in P(\AValGraph[3,321])$ but $\vec{v}\notin P^{321}_k$. 
	
	Since the patterns $(231,312)$ form a simple cycle in $\AValGraph[3,321]$, by definition we get that $\vec{v}\in P(\AValGraph[3,321])$. 
	
	Now assume for sake of contradiction that $\vec{v}\in P^{321}_k$. There exists a sequence
	$(\sigma^m)_{m\in\Z_{\geq 1}} $ in ${\Av(321)}^{\Z_{\geq 1}}$ such that  $|\sigma^m| \to
	\infty$ and $\pcoc(\pi, \sigma^m ) \to \frac{\mathds{1}_{\{231,312\}}(\pi)}{2}$ for all $\pi\in \SS_3$. Consider an interval $I=\{i,i+1,i+2\}$ such that $\pat_{I}(\sigma^m)=312$ and $i+3 \leq |\sigma^{m}|$. Note that since $\sigma^m\in {\Av(321)}$ then $\pat_{\{i+1,i+2,i+3\}}(\sigma^m)\neq 231$; otherwise we would have $\pat_{\{i,i+1,i+3\}}(\sigma^m)=321$. Note also that it is not possible to have $\pat_{\{i+1,i+2,i+3\}}(\sigma^m)\neq 312$ since $\sigma^m(i+1)<\sigma^m(i+2)$. Therefore if $\pat_{I}(\sigma^m)=312$ and $i+3 \leq |\sigma^{m}|$, then $\pat_{\{i+1,i+2,i+3\}}(\sigma^m)\in \{123,213,132,321\}$. This is a contradiction with the fact that 
	\begin{equation*}
		\pcoc(312, \sigma^m ) \to 1/2,\text{ and }
		\pcoc(\pi, \sigma^m ) \to 0, \text{ for all } \pi\in \{123,213,132,321\}.\qedhere
	\end{equation*}
\end{proof}

As a consequence, the feasible region $P^{\monodown_n}_k$ does not coincide with the cycle polytope of the overlap graph $\AValGraph[k,\monodown_n]$. 
In \cref{sect:mon_avoid_case} we introduce a coloured version of the graph $\AValGraph[k,\monodown_n]$, denoted $\OvMon[k, \monodown_n]$, which helps us overcome the problem of the description of the feasible region $P^{\monodown_n}_k$ through a cycle polytope (see in particular \cref{defn:ov_graph_av_monotoneperm}).

The main result for the monotone patterns case is the following one.

\begin{thm}
	\label{thm:feasregmonotones*}
	 Fix $\monodown_n = n \cdots 1 $ for $n\in\Z_{\geq 2}$. There exists a projection map $\Pi$, explicitly described in \cref{eq:feasregmonotones}, such that the consecutive patterns feasible region $P^{\monodown_n}_k$ is the $\Pi$-projection of the cycle polytope of the coloured overlap graph $\OvMon[k, \monodown_n]$. 
	That is,
	$$P^{\monodown_n}_k = \Pi ( P (\OvMon[k, \monodown_n]) )  \, .$$
\end{thm}

An instance of the result stated in \cref{thm:feasregmonotones*} is depicted on the bottom part of \cref{fig:P_3_and_rest}. We remark that \cref{thm:feasregmonotones*} highlights what kind of difficulties can be encountered in proving \cref{conj:dim_is_tight}.

%\begin{thm}\label{thm:dimension}
%	The dimension of $P^{\monodown_n}_k $ is $|\Av_k(\monodown_n)| - |\Av_{k-1}(\monodown_n)|$.
%\end{thm}
%
%\begin{rem}
%	Note that \cref{thm:dim_cycle_pol} gives the dimension of the polytope  $P (\OvMon[k, \monodown_n])$, but one needs to carefully keep track of what happens in the projection in order to determine the dimension of $P^{\monodown_n}_k = \Pi ( P (\OvMon[k, \monodown_n]) )$.
%	A technique for that is developed in \cref{sect:mon_avoid_case}, where we explicitly compute what dimensions are lost after applying the projection to the original polytope.
%\end{rem}
%
%For more information on the numbers $A(n,k)=|\Av_k(\monodown_n)|$ we refer to \cite[A214015]{sloane1996line}. We just recall that a closed formula for these numbers is not available. However, thanks to the Robinson-Schensted correspondence, $A(n,k)$ is equal to $\sum_{\lambda} f_{\lambda}^2$, where the sum runs over all partitions $\lambda$ of $k$ with at most $n-1$ parts and $f_{\lambda}$ is the number of standard Young tableaux with shape $\lambda$.
%
%Note that $A(3,k)=C_k$ is the $k$-th Catalan number. Therefore, Theorems \ref{thm:312-avoiding} and \ref{thm:dimension} imply that $\dim(P^{\Av(\rho)}_k)$ does not depend on the particular permutation $\rho\in\SS_3$.

\subsection{Future projects and open questions}\label {sect:future}

We present here some open questions.

\begin{itemize}
	\item \cref{thm:312-avoiding} and \cref{thm:feasregmonotones*} give a description of the feasible regions $P^{\tau}_k$ for all patterns $\tau$ of size three. Can we describe the feasible regions $P^{B}_k$ for all subsets $B\subseteq \SS_3$? It is easy to see that $P^{B}_k\subseteq \bigcap_{\tau\in B}P^{\tau}_k$, but the reverse inclusion does not hold in general.
	
	\item It seems to be the case that the feasible region $P^{B}_k$ can be precisely described for other specific sets of patterns $B$ different from the ones already considered in this paper. In particular, we believe that a good choice would be a set of (possibly generalized) patterns $B$ for which the corresponding family $\Av(B)$ can be enumerated with generating trees. Indeed, the first author of this article has recently shown in \cite{borga2020asymptotic} that generating trees behave well in the analysis of consecutive patterns of permutations in these families. We believe that generating trees would be particularly helpful to prove some analogues of \cref{lemma:pathpermmono_part2} - that is the key lemma in the proof of \cref{thm:feasregmonotones*} - for other families of permutations.
	
	\item The main open question of this article is \cref{conj:dim_is_tight}. 
\end{itemize}

\subsection{Notation}\label{sect:notation}
We present now some notation and simple results that we will use throughout.

\textbf{Permutations and patterns. }
We recall that we denoted by $\mathcal{S}_n$ the set of permutations of size $n$, and by $\SS$ the set of all permutations.

If $x_1, \dots , x_n$ is a sequence of distinct numbers, let $\std(x_1, \dots,  x_n)$ be the unique permutation $\pi$ in $\SS_n$ whose elements are in the same relative order as $x_1, \dots,  x_n$, i.e.\ $\pi(i)<\pi(j)$ if and only if $x_i<x_j.$
Given a permutation $\sigma\in\SS_n$ and a subset of indices $I\subseteq[n]$, let $\text{pat}_I(\sigma)$ be the permutation induced by $(\sigma(i))_{i\in I}$, namely, $\text{pat}_I(\sigma)\coloneqq\std\left((\sigma(i))_{i\in I}\right).$ 
For example, if $\sigma=24637185$ and $I=\{2,4,7\}$, then $\text{pat}_{\{2,4,7\}}(24637185)=\std(438)=213$. In two particular cases, we use the following more compact notation: for $k\leq |\sigma|$,  $\be_k(\sigma)\coloneqq\text{pat}_{\{1,2,\dots,k\}}(\sigma)$ and $\en_k(\sigma)\coloneqq\text{pat}_{\{|\sigma|-k+1,|\sigma|-k+2,\dots,|\sigma|\}}(\sigma)$.

Given two permutations, $\sigma\in\mathcal{S}_n$ for some $n\in\Z_{\geq 1}$ and $\pi\in\mathcal{S}_k$ for some $k\leq n,$ and a set of indices $I=\{i_1 < \ldots < i_k\}$, we say that $\sigma(i_1) \ldots \sigma(i_k)$ is an \emph{occurrence} of $\pi$ in $\sigma$ if $\pat_I(\sigma)=\pi$ (we will also say that $\pi$ is a \emph{pattern} of $\sigma$). If the indices $i_1, \ldots ,i_k$ form an interval, then we say that $\sigma(i_1) \ldots \sigma(i_k)$ is a \emph{consecutive occurrence} of $\pi$ in $\sigma$ (we will also say that $\pi$ is a \emph{consecutive pattern} of $\sigma$).
We denote intervals of integers as $[n,m]=\{n,n+1,\dots,m\}$ for $n,m\in\Z_{\geq 1}$ with $n\leq m$.

\begin{exmp} 
	The permutation $\sigma=1532467$ contains an occurrence of $1423$ (but no such consecutive occurrences) and a consecutive occurrence of $321$. Indeed $\pat_{\{1,2,3,5\}}(\sigma)=1423$ but no interval of indices of $\sigma$ induces the permutation $1423.$ Moreover, $\pat_{[2,4]}(\sigma)=321.$
\end{exmp}

We denote by $\oc(\pi,\sigma)$ the number of occurrences of a pattern $\pi$ in $\sigma$
and by $\coc(\pi,\sigma)$ the number of consecutive occurrences of a pattern $\pi$ in $\sigma$.
Moreover, we denote by $\widetilde{\oc}(\pi,\sigma)$ (resp.\ by $\widetilde{\coc}(\pi,\sigma)$) the proportion of occurrences (resp.\ consecutive occurrences) of a pattern $\pi\in\SS_k$ in $\sigma\in\SS_n$, that is,
\begin{equation*}
\label{conpatden}
\poc(\pi,\sigma)\coloneqq\frac{\oc(\pi,\sigma)}{\binom{n}{k}}\in[0,1] , \quad \quad \pcoc(\pi,\sigma)\coloneqq\frac{\coc(\pi,\sigma)}{n}\in[0,1]\, .
\end{equation*}

\begin{rem}
	The natural choice for the denominator of the expression in the right-hand side of the equation above should be $n-k+1$ and not $n,$ but we make this choice for later convenience. Moreover, for every fixed $k,$ there is no difference in the asymptotics when $n$ tends to infinity. 
\end{rem}

For a fixed $k\in\Z_{\geq 1}$ and a permutation $\sigma\in\SS$, we let $\poc_k ( \sigma ), \pcoc_k ( \sigma ) \in [0, 1]^{\SS_k}$ be the vectors 
\[\poc_k ( \sigma )\coloneqq \left(\poc(\pi,\sigma)\right)_{\pi\in\SS_k}, \quad \quad \pcoc_k ( \sigma )\coloneqq \left(\pcoc(\pi,\sigma)\right)_{\pi\in\SS_k}\, .\]

We say that $\sigma$ \emph{avoids} $\pi$ if $\sigma$ does not contain any occurrence of $\pi$. We point out that the definition of $\pi$-avoiding permutations refers to occurrences and not to consecutive occurrences. Given a set of patterns $B\subset\mathcal{S},$ we say that $\sigma$ \emph{avoids} $B$ if $\sigma$ avoids $\pi$ for all $\pi\in B$. We denote by $\Av_n(B)$ the set of $B$-avoiding permutations of size $n$ and by $\Av(B)\coloneqq\bigcup_{n\in\Z_{\geq 1}}\Av_n(B)$ the set of $B$-avoiding permutations of arbitrary finite size. The set $\Av(B)$ is often called a \emph{permutation class}.

We also introduce two classical operations on permutations. We denote with $\oplus$ the \emph{direct sum} of two permutations, i.e.\ for $\tau\in\SS_m$ and $\sigma\in\mathcal{S}_n$, 
$$\tau\oplus\sigma=\tau(1)\dots\tau(m)(\sigma(1)+m)\dots(\sigma(n)+m)\, , $$
and we denote with $\oplus_{\ell}\,\sigma$ the direct sum of $\ell$ copies of $\sigma$ (we remark that the operation $\oplus$ is associative). A similar definition holds for the \emph{skew sum} $\ominus$, 
\begin{equation} \label{eq:skew_sum}
	\tau\ominus\sigma=(\tau(1)+n)\dots(\tau(m)+n)\sigma(1)\dots\sigma(n)\, .
\end{equation}
We say that a permutation is $\oplus$-indecomposable (resp.\ $\ominus$-indecomposable) if it cannot be written as the direct sum (resp.\ skew-sum) of two non-empty permutations.
\medskip

\textbf{Directed graphs. }

All graphs, their subgraphs and their subtrees are considered to be directed multigraphs in this paper (and we often refer to them as directed graphs or simply as graphs).
In a directed multigraph $G=(V(G),E(G))$, the set of edges $E(G)$ is a multiset, allowing for loops and parallel edges. 
An edge $e \in E(G) $ is an oriented pair of vertices, $(v, u)$, often denoted by $v \to u$. We write $\st(e)$ for the starting vertex $v$ and $\ar(e)$ for the arrival vertex $u$.
We often consider directed graphs $G$ with labelled edges, and write $\lb( e ) $ for the label of the edge $e\in E(G)$.
In a graph with labelled edges we refer to edges by using their labels.
Given an edge $e \in E(G)$, we denote by $C_{G}(e)$ (for ``set of \textit{continuations} of $e$'') the set of edges $e' \in E(G)$ such that $\st(e') = \ar(e) $.

A \textit{walk} of size $k $ on a directed graph $G$ is a sequence of $k$ edges $(e_1, \dots , e_k)\in E(G)^{k}$ such that for all $i\in [k-1]$, $\ar(e_i)=\st(e_{i+1}).$ 
A walk is a \textit{cycle} if $\st(e_{1}) = \ar(e_k)$.
A walk is a \textit{path} if all the edges are distinct, as well as its vertices, with a possible exception that $\st ( e_1) = \ar ( e_k ) $ may happen.
A cycle that is a path is called a \textit{simple cycle}. 
Given two walks $w=(e_1 , \dots , e_k)$ and $w'=(e'_1 , \dots , e'_{k'})$ such that $\ar(e_k)=\st(e'_1)$, we write $w\bullet w'$ for their concatenation, i.e.\ $w \bullet w'=(e_1 , \dots , e_k , e'_1 , \dots , e'_{k'})$.
For a walk $w$, we denote by $|w|$ the number of \emph{edges} in $w$.

Given a walk $w=(e_1 , \dots , e_k)$ and an edge $e$, we denote by $n_e(w)$ the number of times the edge $e$ is traversed in $w$, i.e.\ $n_e(w) \coloneqq | \{i\leq k|e_i=e\}| $.

%The \emph{incidence matrix} of a directed graph $G$ is the matrix $L(G)$ with rows indexed by $V(G)$, and columns indexed by $E(G)$, such that for any edge $e=v \to u$ with $v \neq u$, the corresponding column in $L(G)$ has $(L(G))_{v, e} = 1$, $ (L(G)))_{u, e} = -1$ and is zero everywhere else. Moreover, if $e=v \to v$ is a loop, the corresponding column in $L(G)$ has zero everywhere.
%
%For instance, we show in \cref{fig:triangraph} a graph $G$ with its incidence matrix $L(G)$.
%
%
%\begin{figure}[htbp]
%	\begin{center}
%		\begin{equation*}
%		G=\begin{array}{ccc}
%		\includegraphics[scale=0.55]{dirgraph.pdf}
%		\end{array},
%		\qquad L(G)=
%		\begin{blockarray}{*{3}{c} l}
%		\begin{block}{*{3}{>{$\footnotesize}c<{$}} l}
%		$e_1$ & $e_2$ & $e_3$ & \\
%		\end{block}
%		\begin{block}{[*{3}{c}]>{$\footnotesize}l<{$}}
%		0 & -1 & 1 & $v_1$ \\
%		1 & 0 & -1 & $v_2$ \\
%		-1 & 1 & 0 & $v_3$ \\
%		\end{block}
%		\end{blockarray}\,.
%		\end{equation*}\\
%		\caption{A graph $G$ with its incidence matrix $L(G)$. \label{fig:triangraph}}
%	\end{center}
%\end{figure}

\section{Topology and dimensions of the consecutive patterns feasible regions}
\label{sec:topopro}

This section is devoted to the proof of \cref{prop:inclusions} and \cref{thm:dim_is_tight}. 

\medskip

%\begin{prop}\label{prop:bnd_dim}
%	Fix $k\in\Z_{\geq 1}$. For any set of patterns $B\subset\mathcal{S}$, we have that $$P^{B}_k\subseteq P(\AValGraph[k,B])\subseteq  P_k. $$
%\end{prop}

We start with \cref{prop:inclusions}, which states that for all sets of patterns $B\subset\mathcal{S}$, the feasible region $P^{B}_k$ satisfies $P^{B}_k\subseteq P(\AValGraph[k,B]) \subseteq P_k$.
Recall the map $W_k$ associating a walk in $\ValGraph[k]$ to each permutation, defined before \cref{lemma:pathperm}.

\begin{proof}[Proof of \cref{prop:inclusions}]
	We start by proving the first inclusion. Consider any point $\vec{v}\in P^{B}_k$, and a corresponding sequence $\left(\sigma^{\ell}  \right)_{\ell\geq 0}\in\Av(B)^{\Z_{\geq 0}}$ such that $\pcoc_k(\sigma^{\ell}) \to \vec{v}$.
	Because $\sigma^{\ell}\in \Av(B)$, we know that for each $\ell $, $W_k(\sigma^{\ell})$ is a walk in $\AValGraph[k,B]$. Using the same method as in the proof that $P_k \subseteq P(\ValGraph[k])$ in \cite[Theorem 3.12]{borga2019feasible}, we can deduce that $\pcoc_k(\sigma^{\ell})$ converges to a point in $P(\AValGraph[k,B])$.
	Specifically, recall that $W(\sigma^{\ell})$ is a walk in the graph $\AValGraph[k,B]$. 
	We can show that the distance between $\pcoc_k(\sigma^{\ell})$ and $P(\AValGraph[k,B])$ goes to zero by decomposing $W(\sigma^{\ell})$ into cycles in $\AValGraph[k,B]$ and a remaining path with negligible size, and so $\vec{v}\in P(\AValGraph[k,B])$.
	Because $\vec{v}$ is generic, it follows that $P^{B}_k \subseteq P(\AValGraph[k,B])$. 
	
	The second inclusion follows from the fact that $\AValGraph[k,B]$ is a subgraph of $\AValGraph[k,]$ and from \cref{thm:main_res_prev_art}.
\end{proof}

We now turn to the proof of \cref{thm:dim_is_tight}.
We start by stating a classical consequence of the fact that $P^{B}_k$ is a set of limit points.
Here we omit the proof: for a similar proof, see \cite[Lemma 3.1]{borga2019feasible}. 

\begin{lem}\label{lem:P_n_is_closed}
Fix $k\in\Z_{\geq 1}$.  For any set of patterns $B \subseteq \SS$, the feasible region $P_k^{B}$ is a closed set.
\end{lem}

For completeness, we include a simple proof of the statement. Recall that we define $\pcoc_k ( \sigma )\coloneqq \left(\pcoc(\pi,\sigma)\right)_{\pi\in\SS_k}$.

\begin{proof}
	It suffices to show that, for any sequence $(\vec{v}_s)_{s\in\Z_{\geq 1}}$ in $P^{B}_k$ such that $\vec{v}_s\to\vec{v}$ for some $\vec{v}\in [0,1]^{\SS_k}$, we have that $\vec{v}\in P^{B}_k$. 
	For all $s\in\Z_{\geq 1}$, consider a sequence of permutations $(\sigma^m_s)_{m\in\Z_{\geq 1}}\in\Av(B)^{\Z_{\geq 1}}$ such that $|\sigma^m_s|\xrightarrow{m\to\infty}\infty$ and $\pcoc_k( \sigma^m_s)\xrightarrow{m\to\infty}\vec{v}_s$, and some index $m( s )$ of the sequence $(\sigma^m_s)_{m\in\Z_{\geq 1}}$ such that for all $m\geq m(s),$
	$$|\sigma^{m}_s|\geq s\quad\text{and}\quad ||\pcoc_k( \sigma^{m}_s)-\vec{v}_s||_2\leq\tfrac{1}{s}\, .$$
	Without loss of generality, assume that $m(s)$ is increasing. For every $\ell\in\Z_{\geq 1}$, define $\sigma^{\ell}\coloneqq\sigma^{m(\ell)}_\ell$.
	It is easy to show that 
	$$|\sigma^\ell|\xrightarrow{\ell\to\infty}\infty \quad\text{and}\quad\pcoc_k(\sigma^\ell)\xrightarrow{\ell\to\infty}\vec{v}\, ,$$ 
	where we use the fact that $\vec{v}_s\to\vec{v}$. Furthermore, by assumption we have that $\sigma^{\ell} \in \Av(B)$. 
	Therefore $\vec{v}\in P^{B}_k$.
\end{proof}

Next we prove the convexity of $P^{B}_k$ stated in \cref{thm:dim_is_tight}.

\begin{prop}\label{prop:convexity_312}
Fix $k\in\Z_{\geq 1}$. Consider a set of patterns $B\subset\mathcal{S}$ such that the class $\Av(B)$ is closed for one of the two operations $\oplus, \ominus$.
Then, the feasible region $P^{B}_k$ is convex.
\end{prop}

\begin{proof}
	We will present a proof for the case where $\Av(B)$ is closed for the $\oplus$ operation, however the arguments hold equally for the $\ominus$ operation.
	
	Since $P^{B}_k$ is a closed set (by \cref{lem:P_n_is_closed}) it is enough to consider rational convex combinations of points in $P^{B}_k$, i.e.\ it is enough to establish that for all $\vec{v}_1,\vec{v}_2\in P^{B}_k$ and all $ s,t\in\Z_{\geq 1}$, we have that
	\begin{equation*}
		\frac{s}{s+t}\vec{v}_1+\frac{t}{s+t}\vec{v}_2\in P^{B}_k.
	\end{equation*}
	Fix $\vec{v}_1,\vec{v}_2\in P^{B}_k$ and $s,t\in\Z_{\geq 1}$. Since $\vec{v}_1,\vec{v}_2\in P^{B}_k$, there exist two sequences $(\sigma^\ell_1)_{\ell\in\Z_{\geq 1}}$, $(\sigma^\ell_2)_{\ell\in\Z_{\geq 1}}$ such that $|\sigma^\ell_i|\xrightarrow{\ell\to\infty}\infty$, $\sigma^\ell_i \in \Av(B) $ and $\pcoc_k( \sigma^\ell_i)\xrightarrow{\ell\to\infty}\vec{v}_i$, for $i=1,2$.
	
	Define $t_\ell\coloneqq t\cdot |\sigma^\ell_1|$ and $s_\ell\coloneqq s\cdot |\sigma^\ell_2|$. 
%	\begin{figure}[htbp]
%		\begin{center}
%			\includegraphics[scale=.5]{schema_direct_sum}\\
%			\caption{Schema for the definition of the permutation $\tau^\ell$. \label{schema_direct_sum}}
%		\end{center}
%	\end{figure}	
	We set $\tau^\ell\coloneqq \left(\oplus_{s_\ell}\,\sigma^\ell_1\right)\oplus\left(\oplus_{t_\ell}\,\sigma^\ell_2\right)$. %For a graphical interpretation of this construction we refer to \cref{schema_direct_sum}.
	We note that for every $\pi\in\SS_k$, we have
	\begin{equation*}
		\coc(\pi,\tau^\ell)=s_\ell\cdot\coc(\pi,\sigma^\ell_1)+t_\ell\cdot\coc(\pi,\sigma^\ell_2)+Er,
	\end{equation*}
	where $Er\leq(s_\ell+t_\ell-1)\cdot |\pi|$. This error term comes from the number of intervals of size $|\pi|$ that intersect the boundary of some copies of $\sigma^\ell_1$ or $\sigma^\ell_2$. Hence
	\begin{equation*}
	\begin{split}
	\pcoc(\pi,\tau^\ell)&=\frac{s_\ell\cdot|\sigma^\ell_1|\cdot\pcoc(\pi,\sigma^\ell_1)+t_\ell\cdot|\sigma^\ell_2|\cdot\pcoc(\pi,\sigma^\ell_2)+Er}{s_\ell\cdot|\sigma^\ell_1|+t_\ell\cdot |\sigma^\ell_2|}\\
	&=\frac{s}{s+t}\pcoc(\pi,\sigma^\ell_1)+\frac{t}{s+t}\pcoc(\pi,\sigma^\ell_2)+O\left(|\pi|\left(\tfrac{1}{|\sigma^\ell_1|}+\tfrac{1}{|\sigma^\ell_2|}\right)\right).
	\end{split}
	\end{equation*}
	As $\ell$ tends to infinity, we have $$\pcoc_k(\tau^\ell)\to\frac{s}{s+t}\vec{v}_1+\frac{t}{s+t}\vec{v}_2,$$ 
	since $|\sigma^\ell_i|\xrightarrow{\ell\to\infty}\infty$ and $\pcoc_k( \sigma^\ell_i)\xrightarrow{m\to\infty}\vec{v}_i$, for $i=1,2$. Noting also that
	$|\tau^\ell|\to\infty,$
	we can conclude that $\tfrac{s}{s+t}\vec{v}_1+\tfrac{t}{s+t}\vec{v}_2\in P^{B}_k$. This ends the proof of the first part of the statement. 
\end{proof}

We now prove a result that gives an upper bound on the dimension of $P_k^B$.

\begin{prop}\label{prop:strongly_connected}
Fix $k\in\Z_{\geq 1}$ and a set of patterns $B\subset\mathcal{S}$ such that the class $\Av(B)$ is closed for one of the two operations $\oplus, \ominus$.
Then the graph $\AValGraph[k,B] $ is strongly connected and $\dim(P(\AValGraph[k,B])) =  |\Av_k(B)|-|\Av_{k-1}(B)|$.
\end{prop}

\begin{proof}
Consider $v_1, v_2$ two vertices of $\AValGraph[k,B]$, and assume that $\Av(B)$ is closed for $\oplus$, for simplicity.
Then $\lb(v_1) \oplus \lb(v_2)$ is a permutation in $\Av(B)$, so $W_k(\lb(v_1) \oplus \lb(v_2))$ is a walk in the graph $\AValGraph[k,B]$ that connects $v_1$ to $v_2$.
We conclude that $ \AValGraph[k,B] $ is strongly connected.
It follows from \cref{thm:dim_cycle_pol} that 
\begin{equation*}
	\dim(P(\AValGraph[k,B])) =  |\Av_k(B)|-|\Av_{k-1}(B)|.\qedhere
\end{equation*}
\end{proof}

We now fix a set of patterns $B\subset\mathcal{S}$ such that the class $\Av(B)$ is closed under the $\oplus$ operation (the other case is similar). Note that thanks to Propositions \ref{prop:inclusions} and \ref{prop:strongly_connected} we have that 
$\dim(P^{B}_k)\leq |\Av_k(B)|-|\Av_{k-1}(B)|$.

In order to prove \cref{thm:dim_is_tight}, it remains to show that
\begin{equation}\label{eq:goal_of_the_article}
	\dim(P^{B}_k)\geq |\Av_k(B)|-|\Av_{k-1}(B)|.
\end{equation}
Our strategy to prove \cref{eq:goal_of_the_article} is to show that there exists a portion of the polytope $P(\AValGraph[k,B])$  of full dimension $|\Av_k(B)|-|\Av_{k-1}(B)|$ that is contained in $P^{B}_k$. 
We start by explicitly describing this portion. 

Recall first that from \cite[Proposition 2.2]{borga2019feasible}, if $G=(V,E)$ is a directed multigraph then the vertices of the polytope $P(G) $ are precisely the vectors $$\{\vec{e}_{\mathcal{C}} | \, \mathcal{C} \text{ is a simple cycle of } G \}.$$

Consider the vertex $\vec{e}_{\ell}$ of $P(\AValGraph[k,B])$ corresponding to the loop $\ell$ in $\AValGraph[k, B]$ given by the increasing permutation $\iota_k=1\dots k$ (here we are using the fact that $B$ is closed under $\oplus$).

Let $\mathcal N \mathcal E_k$ be the set of permutations $\sigma$ in $\Av_k(B)$ such that $\sigma(k)\neq k$ ($\mathcal N \mathcal E_k$ stands for \emph{not ending with} $k$ but also recalls that the permutations in $\mathcal N \mathcal E_k$ have size $k$). 
For $\pi\in \mathcal N \mathcal E_k$, we set 
$$\sigma_n(\pi)\coloneqq\oplus_n(\pi\oplus\iota_{k}),\quad\text{for all}\quad n\in\Z_{>0}\qquad \text{and}\qquad\vec{p}(\pi)\coloneqq\lim_{n\to \infty}\pcoc_k(\sigma_n(\pi)).$$

We show that the limit is well defined.

\begin{lem}\label{lem:pos_entr}
	For every $\pi\in \mathcal N \mathcal E_k$, the limit $\vec{p}(\pi)=\lim_{n\to \infty}\pcoc_k(\sigma_n(\pi))$ exists and it satisfies
	$$\vec{p}(\pi)=\frac{\coc_{k}(\pi\oplus\iota_{k}\oplus\pat_{k-1}(\pi))}{2k}.$$	
	In particular, $(\vec{p}(\pi))_{\pi}\neq0$.
	%	there exist $\alpha_\pi,\beta_\pi\in(0,1)$ %with $\alpha_\pi + \beta_\pi = 1$ 
	%	such that
	%	$$\vec{p}(\pi)=\alpha_\pi \vec{e}_{\pi} + \beta_\pi \vec{e}_{\ell}+\sum_{\rho\in\mathcal S_k\setminus\{\pi,\ell\}}\gamma_{\pi \vec{e}_{\rho}$$
\end{lem}

\begin{proof}
	Note that for any $\rho\in\mathcal S_k$,
	\begin{multline*}
		\pcoc(\rho,\sigma_n(\pi))=\pcoc(\rho,\oplus_n(\pi\oplus\iota_{k}))\\
		=\frac{(n-1)\cdot\coc(\rho,\pi\oplus\iota_{k}\oplus\pat_{k-1}(\pi))}{2k\cdot n}+\frac{\coc(\rho,\pi\oplus\iota_{k})}{2k\cdot n}\\
		\to\frac{\coc(\rho,\pi\oplus\iota_{k}\oplus\pat_{k-1}(\pi))}{2k}.
	\end{multline*}
	See \cref{fig:pi} to clarify the decomposition of $\pcoc(\rho,\oplus_n(\pi\oplus\iota_{k}))$.
	Obviously if $\rho=\pi$ then $\coc(\rho,\pi\oplus\iota_{k}\oplus\pat_{k-1}(\pi))\geq 1$ and so $(\vec{p}(\pi))_{\pi}\neq0$.
\end{proof}

\begin{figure}
	\begin{minipage}[c]{0.49\textwidth}
		\centering
		\includegraphics[scale=0.25]{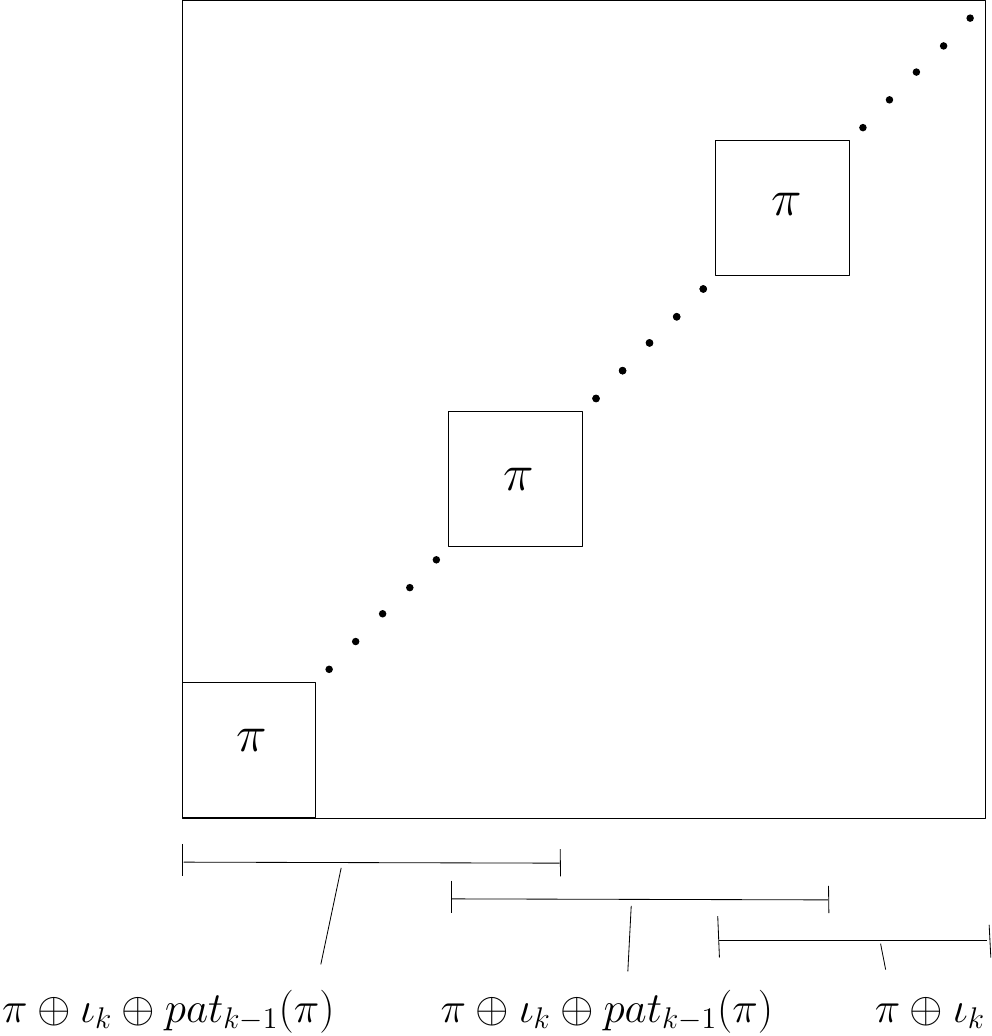}
	\end{minipage}\hfill
	\begin{minipage}[c]{0.5\textwidth}
		\captionsetup{width=\textwidth}
		\caption{Any consecutive pattern of size $k$ of $\sigma^n(\pi)$  will fit in exactly one of the highlighted intervals of indexes.\label{fig:pi}}
	\end{minipage}
\end{figure}

The following proposition describes the portion of $P(\AValGraph[k,B])$ that is contained in $P^{B}_k$.

\begin{prop}\label{prop:portion_contained} 
	The polytope $\conv\left(\{\vec{p}(\pi):\pi\in \mathcal N \mathcal E_k\}\cup\{\vec{e}_{\ell}\}\right)$ is contained inside $P^{B}_k$.	
\end{prop}

\begin{proof}
	Since $P^{B}_k$ is convex thanks to \cref{prop:convexity_312}, it is enough to show that 
	\begin{itemize}
		\item $\vec{e}_{\ell} \in P^{B}_k$.
		\item $\vec{p}(\pi)\in P^{B}_k$, for every  $\pi\in \mathcal N \mathcal E_k$.
	\end{itemize}
	The first claim follows from the fact that $B$ is closed under $\oplus$ and therefore the increasing permutations $(\iota_m)_{m\in\Z_{>0}}$ avoid $B$ and satisfy  $\pcoc_k(\iota_m) \to \vec{e}_{\ell}$.
	For the second claim, it is enough to note that $\sigma_m(\pi)=\oplus_m(\pi\oplus\iota_{k})\in\Av(B)$ for all $m\in\Z_{>0}$ (where we are using again that $B$ is closed under $\oplus$) and recall the definition of $\vec{p}(\pi)$.
\end{proof}

%\begin{lm}
%%Let $\mathcal N \mathcal E_k$ be the set of permutations $\tau$ in $\Av_k(\pi)$ such that $\tau(k) \neq k$.
%The polytope resulting from the convex hull of $\vec{e}_{\ell}$ and $\{\alpha_e \vec{e}_{\CCC_e} + \beta_e \vec{e}_{\ell}\}_{e \in \mathcal N \mathcal E_k}$, where the coefficients $\{\alpha_e\}_{e \in \mathcal N \mathcal E_k} $ and $\{\beta_e\}_{e \in \mathcal N \mathcal E_k} $ are arbitrary such that $\alpha_e + \beta_e = 1 $ and $0 < \alpha_e \leq 1$ has dimension at least $|\mathcal N \mathcal E_k|$.
%\end{lm}

The following proposition guarantees that the polytope 
$$P\coloneqq\conv\left(\{\vec{p}(\pi):\pi\in \mathcal N \mathcal E_k\}\cup\{\vec{e}_{\ell}\}\right)$$
has the correct dimension that we need to prove \cref{eq:goal_of_the_article}.

\begin{prop}\label{prop:portion_dimension}
	Let $P = \conv\left(\{\vec{p}(\pi):\pi\in \mathcal N \mathcal E_k\}\cup\{\vec{e}_{\ell}\}\right)$. The following lower bound holds 
	\begin{equation}
		\dim\left( P \right)\geq|\Av_k(B)|-|\Av_{k-1}(B)|.
	\end{equation}
\end{prop}

\begin{proof}
	We start by defining a partial order $\prec $ on $\mathcal N \mathcal E_k$.
	For $\tau_1, \tau_2\in \mathcal N \mathcal E_k $ we say that $\tau_1 \prec \tau_2$  if $\tau_2 = \pat_{ [ 1 , k ] }(\iota_m \oplus \tau_1) $ for some $m\in\Z_{\geq 0}$.
	This relation is clearly transitive and reflexive.
	
	To observe that it is also anti-symmetric notice that if $\tau_1 \prec \tau_2$ and $\tau_2 \prec \tau_1$ then there exist two integers $m_1, m_2\in\Z_{\geq 0}$ such that $\tau_2 = \pat_{ [ 1 , k ] }(\iota_{m_1} \oplus \tau_1 )$ and $\tau_1 = \pat_{ [ 1 , k ] }(\iota_{m_2} \oplus \tau_2 )$.
	Hence $\tau_2 =\pat_{ [ 1 , k ] }(\iota_{m_1+m_2} \oplus \tau_2 )$. Now one can see that if $m_1+m_2> 0$ then the only solution to this equation is $\tau_2=\iota_k$, but this is not possible because $\tau_2 \in \mathcal N \mathcal E_k$. Therefore $m_1=m_2=0$ and so $\tau_1=\tau_2$.
	
	Thus, $\prec$ defines a partial order.
	
	\medskip
	
	Now, consider the collection of linear functionals $f_\rho\in \left(\mathbb{R}^{\mathcal S_k}\right)^*$ defined for all $\rho\in\mathcal{S}_k$ by 
	$$f_\rho(\vec{e}_{\pi}) =\delta_{\rho, \pi},\quad\text{for all}\quad \pi\in\mathcal{S}_k,$$ 
	where $\delta$ denotes the Kronecker delta.
	We also fix an (arbitrary) extension of the partial order $\prec$ to a total order and we define the following matrix
	$$A = \Big( f_\rho\left(\vec{p}(\pi)  - \vec{e}_{\ell} \right) \Big)_{\rho,\pi \in \mathcal N \mathcal E_k} =  \Big( f_\rho\left(\vec{p}(\pi) \right) \Big)_{\rho,\pi \in \mathcal N \mathcal E_k}.$$
	Because $(\vec{p}(\pi))_{\pi\in\mathcal N \mathcal E_k},  \vec{e}_{\ell} $ are in $P$ and so also in the affine span of $P$, we have that the dimension of $P$ is bounded below by the dimension of 
	$$\mathrm{span}\{\vec{p}(\pi)  - \vec{e}_{\ell} |  \pi \in \mathcal N \mathcal E_k \}\, .$$
	This dimension is bounded below by the rank of $A$.
	It suffices then to show that $A$ is upper-triangular with non-zero elements in the diagonal, showing that it is full rank. Indeed, using that $|\mathcal N \mathcal E_k|=|\Av_k(B)|-|\Av_{k-1}(B)|$, this would conclude the proof.
	
	%Let us show that $A$ is upper-triangular with non-zero elements in the diagonal.
	
	First, on the diagonal, we have that $ f_\pi(\vec{p}(\pi))  = (\vec{p}(\pi))_\pi \neq 0$ by \cref{lem:pos_entr}.
	
	On the other hand, if $\rho, \pi \in \mathcal N \mathcal E_k$, $\rho\neq \pi$, and $f_\rho(\vec{p}(\pi))=(\vec{p}(\pi))_\rho$ is non-zero, then we must have $\coc(\rho,\pi\oplus\iota_{k}\oplus\pat_{k-1}(\pi))\geq 1$, but because $\rho\in \mathcal N \mathcal E_k$, it is immediate to see that there exists some $m\in\Z_{>0}$ such that $\rho = \pat_{[1, k]}(\iota_m \oplus \pi)$ and so $\pi\prec \rho $.
\end{proof}

%$$A = \Big( f_e\left(\alpha_{e'} \vec{e}_{\CCC_{e'}} + (\beta_{e'}  - 1 )\vec{e}_{\ell} \right) \Big)_{e, e' \in \mathcal N \mathcal E_k}.$$
%Because $\alpha_e \vec{e}_{\CCC_e} + \beta_e\vec{e}_{\ell} $ and $\vec{e}_{\ell} $ are vertices of the polytope\todo{Not needed right?}, we have that the dimension of the polytope is bounded below by the rank of $A$.
%It suffices then to show that $A$ is upper-triangular with non-zero elements in the diagonal, showing that it is full rank.
%
%Indeed, if we extend $\prec$ to a total order, we will have in $A$ an upper-triangular matrix.
%
%First, noting that $e \in \CCC_e$, we have that $ f_e(\alpha_e \vec{e}_{\CCC_e} + (\beta_e  - 1 )\vec{e}_{\ell} )  = \frac{\alpha_e }{|\CCC_e|} \neq 0$ because $\ell=1\dots k\notin N \mathcal E_k$ and $(\vec{e}_{\CCC_e})_e=\frac{1}{|\CCC_e|}$ since $\CCC_e$ is a simple cycle containing $e$.
%
%On the other hand, if $e, e' \in \mathcal N \mathcal E_k$, $e\neq e'$, and $f_e(\alpha_{e'} \vec{e}_{\CCC_{e'}} + (\beta_{e'}  - 1 )\vec{e}_{\ell} )$ is non-zero, then we must have $e' \in \CCC_e$, but because $e\in \mathcal N \mathcal E_k$ it is imediate to see that there exists some $m\in\Z_{>0}$ such that $e = \pat_{[1, k]}(1^{\oplus m} \oplus e')$.

Propositions \ref{prop:portion_contained} and \ref{prop:portion_dimension} prove \cref{eq:goal_of_the_article} and complete the proof of \cref{thm:dim_is_tight}.

%\begin{proof}[Proof of \cref{thm:dim_is_tight}]
%	Recall that thanks to \cref{prop:inclusions}, it is enough to prove that
%	\begin{equation}
%		\dim(P^{\tau}_k)\geq |\Av_k(\tau)|-|\Av_{k-1}(\tau)|.
%	\end{equation}
%	This bound follows from Propositions \ref{prop:portion_contained} and \ref{prop:portion_dimension}, where we showed that
%	\begin{equation}
%		\conv\left(\{\vec{p}(\pi):e\in \mathcal N \mathcal E_k\}\cup\{\vec{e}_{\ell}\}\right)\subseteq P^{\tau}_k,
%	\end{equation}
%	and that
%	\begin{equation}
%		\dim\left(\conv\left(\{\vec{p}(\pi):e\in \mathcal N \mathcal E_k\}\cup\{\vec{e}_{\ell}\}\right)\right)=|\Av_k(\tau)|-|\Av_{k-1}(\tau)|.
%	\end{equation}
%\end{proof}

\section{The feasible region for 312-avoiding permutations}\label{sect:312-avoiding}

This section is devoted to the proof of \cref{thm:312-avoiding}. 
The key step in this proof is to show an analogue of \cref{lemma:pathperm} for $312$-avoiding permutations. More precisely, we have the following.

\begin{lem}\label{lemma:pathperm312}
	Fix $k\in\Z_{\geq 1}$ and $m\geq k$. The map $W_k$, from the set $\Av_{m}(312)$ of 312-avoiding permutations of size $m$ to the set of walks in $\AValGraph[k,312]$ of size $m-k+1$, is surjective.	
\end{lem}

To prove the lemma above we have to introduce the following.

\begin{defn}\label{defin:append_permutation}
	Given a permutation $\sigma\in\SS_n$ and an integer $\ell\in [n+1],$ we denote by $\sigma^{*\ell}$ the permutation obtained from $\sigma$ by \emph{appending a new final value} equal to $\ell$ and shifting by $+1$ all the other values larger than or equal to $\ell.$ %Equivalently,
%	$$\sigma^{*\ell}\coloneqq\text{std}(\sigma(1),\dots,\sigma(n),\ell-1/2).$$
\end{defn}

The proof of \cref{lemma:pathperm312} is based on the following result. Recall the definition of the set $C_{G}(e)$ of continuations of an edge $e$ in a graph $G$, i.e.\ the set of edges $e' \in E(G)$ such that $\st(e') = \ar(e) $.

\begin{lem}\label{lemma:lem_for_pathperm312}
	Let $\sigma$ be a permutation in $\Av(312)$ such that $\en_k(\sigma)=\pi$ for some $\pi\in\Av_k(312)$. Let $\pi'\in\Av_k(312)$ such that $\pi'\in C_{\AValGraph[k,312]}(\pi)$. Then there exists $\ell\in[|\sigma|+1]$ such that $\sigma^{*\ell}\in\Av(312)$ and $\en_k(\sigma^{*\ell})=\pi'$.
\end{lem}

We first explain how \cref{lemma:pathperm312} follows from \cref{lemma:lem_for_pathperm312} and then we prove the latter.

\begin{proof}[Proof of \cref{lemma:pathperm312}]
	In order to prove the claimed surjectivity, given a walk $w=(e_1,\dots,e_s)$ in $\AValGraph[k,312]$, we have to exhibit a permutation $\sigma\in\Av(312)$ of size $s+k-1$ such that $W_k(\sigma)=w$.
	We do that by constructing a sequence of $s$ permutations $(\sigma_i)_{i\leq s}\in{(\Av(312))^s}$ with size $|\sigma_i|=i+k-1$, in such a way that $\sigma$ is equal to $\sigma_s$. Moreover, we will have that $\be_{|\sigma_{i+1}|-1}(\sigma_{i+1})=\sigma_i$.
	
	The first permutation is defined as $\sigma_1 = \lb (e_1 )$. To construct $\sigma_{i+1}$ from $\sigma_{i}$, note that from \cref{lemma:lem_for_pathperm312} there exists $\ell\in[|\sigma_{i}|+1]$ such that $\en_k(\sigma_{i}^{*\ell})$ is equal to the pattern $\lb ( e_{i+1} )$ and $\sigma_{i}^{*\ell}$ avoids the pattern 312. Then we define $\sigma_{i+1}\coloneqq\sigma_{i}^{*\ell}$, determining the sequence $(\sigma_i)_{i\leq s}\in{(\Av(312))^s}$. Finally, setting $\sigma\coloneqq \sigma_s$ we have by construction that $W_k(\sigma)=w$ and that $\sigma\in\Av_{s+k-1}(312)$. 
\end{proof}

\begin{proof}[Proof of \cref{lemma:lem_for_pathperm312}]
We have to distinguish two cases.

\textbf{Case 1:} $\pi'(k)\in\{1,k\}$. We define $\ell\coloneqq \mathds{1}_{\{\pi'(k)=1\}}+(|\sigma|+1)\mathds{1}_{\{\pi'(k)=k\}}$. In this case one can see that $\sigma^{*\ell}\in\Av(312)$ -- the new final value $\ell$ cannot create an occurrence of $312$ in $\sigma^{*\ell}$ -- and that $\en_k(\sigma^{*\ell})=\pi'$. 

\textbf{Case 2:} $\pi'(k)\in[2,k-1]$. Consider the point just above $(k,\pi'(k))$ in the diagram of $\pi'$ and the corresponding point in the last $k-1$ points of $\sigma$ (for an example see the two red points in \cref{fig:Schema_proof_312}). Let $i$ be the index in the diagram of $\sigma$ of the latter point.
We claim that $\sigma^{*\sigma(i)}\in\Av(312)$ and $\en_k(\sigma^{*\sigma(i)})=\pi'$. The latter is immediate.  
It just remains to show that $\sigma'\coloneqq\sigma^{*\sigma(i)}\in\Av(312)$.

Assume by contradiction that $\sigma'$ contains an occurrence of $312$. Since by assumption $\sigma\in\Av(312)$ then the value $2$ of the occurrence $312$ must correspond to the final value $\sigma'(|\sigma'|)=\sigma(i)$ of $\sigma'$. Moreover, since $\pi'\in\Av(312)$, the $312$-occurrence cannot occur in the last $k$ elements of $\sigma'$, that is the $312$-occurrence must occur at the values $\sigma'(j),\sigma'(r),\sigma'(|\sigma'|)$ for some indices $j\leq |\sigma'|-k$ and $j<r<|\sigma'|$. Because $\sigma'(j),\sigma'(r),\sigma'(|\sigma'|)$ is an occurrence of 312, $\sigma'(j)>\sigma'(|\sigma'|)$. Moreover, since $\sigma'(i)=\sigma'(|\sigma'|)+1$ by construction, it follows that $\sigma'(j)>\sigma'(i)$. Note that $r\neq i$ since $\sigma'(i)=\sigma'(|\sigma'|)+1$ and $\sigma'(r)<\sigma'(|\sigma'|)$.
Therefore, we have two cases:
\begin{itemize}
	\item If $r<i$ then $\sigma'(j),\sigma'(r),\sigma'(i)$ is also an occurrence of $312$. A contradiction to the fact that $\sigma\in\Av(312)$.
	\item If $r>i$ then $\sigma'(i),\sigma'(r),\sigma'(|\sigma'|)$  is also an occurrence of $312$. A contradiction to the fact that $\pi'\in\Av(312)$.
\end{itemize}
This concludes the proof.
\end{proof}

\begin{figure}[htbp]
  \centering
     \includegraphics[scale=.5]{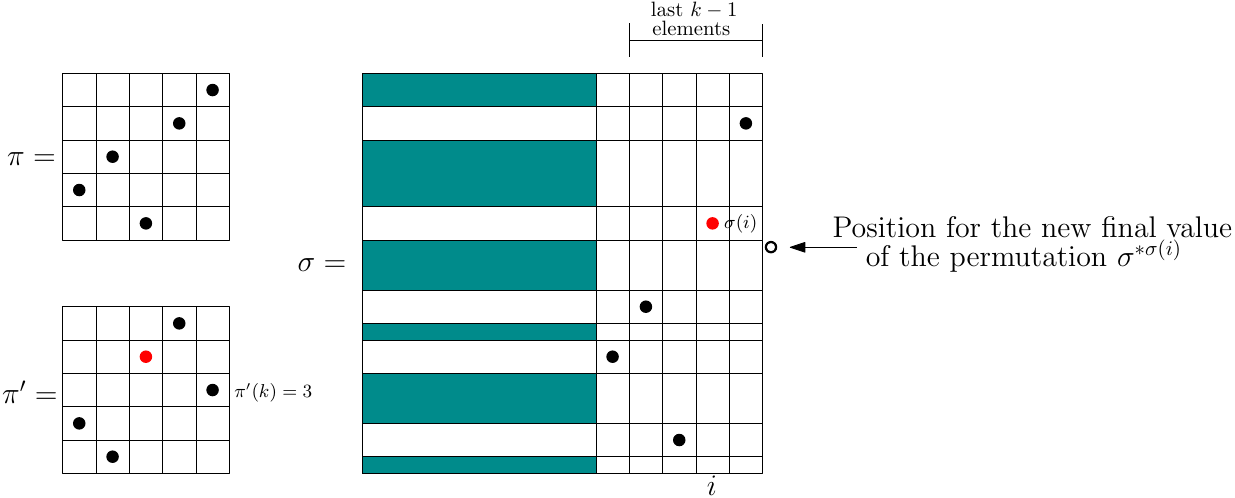}\\
  \caption{A schema for the proof of \cref{lemma:lem_for_pathperm312}. \label{fig:Schema_proof_312}}
\end{figure}

Building on \cref{prop:convexity_312} and \cref{lemma:pathperm312} we can now prove \cref{thm:312-avoiding}.

\begin{proof}[Proof of \cref{thm:312-avoiding}]
	The fact that $P^{312}_k=P(\AValGraph[k,312])$ follows using exactly the same proof of \cite[Theorem 3.12]{borga2019feasible} replacing Lemma 3.8 and Proposition 3.2 of \cite{borga2019feasible} by \cref{lemma:pathperm312} and \cref{prop:convexity_312} of this paper (note that in the proof of  \cite[Theorem 3.12]{borga2019feasible} we also use the fact that the feasible region is closed and this is still true for $P^{312}_k$, thanks to \cref{lem:P_n_is_closed}).
	%The fact that the dimension of $P^{312}_k$ is $C_k - C_{k-1}$ follows from \cref{prop:strongly_connected} and the well-known fact that the number of permutations of size $k$ avoiding the pattern $312$ is equal to the $k$-th Catalan number. Finally the fact that the vertices of $P^{312}_k$ are given by the simple cycles of $\AValGraph[k,312]$ is a consequence of \cite[ Proposition 2.2]{borga2019feasible}.
\end{proof}

\section{The feasible region for monotone-avoiding permutations}\label{sect:mon_avoid_case}

Fix $\monodown_n = n \cdots 1 $, the decreasing pattern of size $n\in\Z_{\geq 1}$.
In this section we study $P^{\monodown_n}_k $ and we show that it is related to the cycle polytope of the coloured overlap graph $ \OvMon[k,\monodown_n]$ , presented in \cref{defn:ov_graph_av_monotoneperm} -- this is \cref{thm:feasregmonotones*}, more precisely restated in \cref{thm:feasregmonotones}.
%We also compute the dimension of $P^{\monodown_n}_k $ -- this is \cref{thm:dimension}.

\subsection{Definitions and combinatorial constructions}

We start by introducing colourings of permutations.

\begin{defn}[Colourings and RITMO colourings]
	\label{defn:colourings_of_permutations}
	Fix an integer $m\in\Z_{\geq 1}$.
	For a permutation $\sigma$, an $m$-colouring of $\sigma$ is a map $\mathfrak{c}: [|\sigma|] \to [m]$, which is to be interpreted as a map from the set of \emph{indices} of $\sigma$ to $[m]$.
	An $m$-colouring $\mathfrak{c}$ is said to be \emph{rainbow} when $\im (\mathfrak{c}) = [m]$.
	For any permutation $\sigma$, we define its \emph{right-top monotone colouring} (simply RITMO colouring henceforth), which we denote as $\mathbb{C}(\sigma)$.
	This colouring is constructed iteratively, starting with the highest value of the permutation which receives the colour 1 and going down while assigning the lowest possible colour that prevents the occurrence of a monochromatic 21.
\end{defn}	

If a permutation is coloured with its RITMO colouring, the left-to-right maxima are coloured by $1$; removing these left-to-right maxima, the left-to-right maxima of the resulting set of points are coloured by $2$, and so on. We suggest to the reader to keep in mind both points of view (the one given in the definition and the one described now) on RITMO colourings.

\begin{exmp}
	In all our examples, we paint in {\color{red} red} the values coloured by $1$, in {\color{blue} blue} the ones coloured by two, and in {\color{green} green} the ones coloured by three.
	For instance, the RITMO colouring for permutations $1427536$ is given by $\color{red}{1}\color{red}{4}\color{blue}{2}\color{red}{7}\color{blue}{5}\color{green}{3}\color{blue}{6}$. 
\end{exmp}

%\begin{figure}[htbp]
%	\centering
%	\includegraphics[scale=0.7]{RITMO_colourings_of_permutations.pdf}
%	\caption{\label{fig:RITMO_colouring_of_permutations} The RITMO colourings of $312$, $1427536$ and $124376985$.}
%\end{figure}

For the pair $(\sigma, \mathbb{C}(\sigma))$ we simply write $\mathbb{S}(\sigma)$.
If $\sigma $ avoids the permutation $\monodown_n $, it is known that its RITMO colouring is an $(n-1)$-colouring (the origins of this result are hard to trace, but it goes back at least to \cite{greene1974extension} where it is already noted as something that is not hard to prove; see also \cite[Chapter 4.3]{MR2919720}).

We furthermore allow for taking \textit{restrictions of colourings}.
Given a permutation $\sigma $ of size $k$, a colouring $\mathfrak{c}$ of $\sigma $ and a subset $I=\{i_1, \dots ,i_j\}\subseteq [k]$, we consider the restriction $\pat_I(\sigma,\mathfrak{c})$ to be the pair $(\pat_I(\sigma),\mathfrak{c}')$, where $\mathfrak{c}'(\ell)=\mathfrak{c}(i_{\ell})$ for all $\ell\in [j]$.

The following definition is fundamental in our results.

\begin{defn}
	We say that an $m$-colouring $\mathfrak{c}$ of a permutation $\pi\in\Av(\monodown_n)$ of size $k$ is \emph{inherited} if there is some permutation $\sigma\in\Av(\monodown_n)$ of size $\ell \geq k$ such that $\en_k ( \mathbb{S}(\sigma) )=(\pi, \mathfrak{c})$.
\end{defn}
Observe that it may be the case that $\pat_I(\mathbb{S}(\sigma))$ and $\mathbb{S}(\pat_I(\sigma))$ are distinct inherited colourings of the permutation $\pat_I(\sigma)$. For instance, if $\sigma=2134$ and $I=\{2,3,4\}$ then  $\pat_I(\mathbb{S}(\sigma))=\pat_{\{2, 3, 4\}}({\color{red}2}{\color{blue}1}{\color{red}34})={\color{blue}1}{\color{red}23}$ but $\mathbb{S}(\pat_I(\sigma))=\mathbb{S}(123)={\color{red}123}$.
This is unlike the relation between $\mathbb{S}$ and $\be$, as one can see in \cref{obs:Correctness}.

To sum up, we have introduced three notions of colourings, each more restricted than the previous one. In particular, any RITMO colouring is an inherited colouring, and any inherited colouring is a colouring.

\medskip

Let $\CCC_{m}(\pi )$ be the set of all inherited $m$-colourings of a permutation $\pi\in\Av(\monodown_n)$. We also set $\CCC_{m}(k) =\{ (\pi, \mathfrak{c}) | \pi\in \Av_k(\monodown_n), \, \, \mathfrak{c} \text{ is an inherited $m$-colouring of } \pi \}$, that is the set of all inherited $m$-colourings of permutations of size $k$.

\begin{exmp}
	Let $n = 3$. In \cref{tbl:inherited_2_colourings} we present all the inherited $2$-colourings of permutations of size three. Thus, 
	$$ \mathcal C_2(3) = \left\{{\color{red}123}, {\color{blue}1}{\color{red}23}, {\color{blue}12}{\color{red}3}, {\color{blue}123}, {\color{red}13}{\color{blue}2}, {\color{blue}1}{\color{red}3}{\color{blue}2}, {\color{red}2}{\color{blue}1}{\color{red}3}, {\color{red}23}{\color{blue}1}, {\color{red}3}{\color{blue}12}\right\} \, .$$
\end{exmp}
\begin{table}[htbp]
	\centering
	\begin{tabular}{|c | l |}
		\hline
		123 & ${\color{red}123}=\mathbb{S}(123)$, ${\color{blue}1}{\color{red}23}=\en_3({\color{red}2}{\color{blue}1}{\color{red}34})$, ${\color{blue}12}{\color{red}3}=\en_3({\color{red}3}{\color{blue}12}{\color{red}4})$, ${\color{blue}123}=\en_3({\color{red}4}{\color{blue}123})$ \\
		\hline
		132 &  ${\color{red}13}{\color{blue}2}=\mathbb{S}(132)$, ${\color{blue}1}{\color{red}3}{\color{blue}2}=\en_3({\color{red}3}{\color{blue}1}{\color{red}4}{\color{blue}2})$ \\ 
		\hline
		213 & ${\color{red}2}{\color{blue}1}{\color{red}3}=\mathbb{S}(213)$ \\
		\hline
		231 & ${\color{red}23}{\color{blue}1}=\mathbb{S}(231)$ \\
		\hline
		312 & ${\color{red}3}{\color{blue}12}=\mathbb{S}(312)$ \\
		\hline
	\end{tabular}
	\captionsetup{width=\textwidth}
	\caption{The permutations of size three, and their corresponding inherited $2$-colourings. Note that all permutations of size four in this table are coloured according to their RITMO colouring. Observe also that the coloured permutation ${\color{red} 2}{\color{blue}13}$ is not inherited.
		\label{tbl:inherited_2_colourings}
	}
\end{table}
We introduce a key definition for this and the consecutive sections.
\begin{defin}
	\label{defn:ov_graph_av_monotoneperm}
	The \emph{coloured overlap graph} $\OvMon[k, \monodown_n]$ is defined with
	the vertex set 
	$$V \coloneqq \CCC_{n-1}(k-1) = \{(\pi , \mathfrak{c}) | \pi \in \Av_{k-1}(\monodown_n), \, \mathfrak{c} \text{ is an inherited $(n-1)$-colouring of } \pi\},$$
	and the edge set
	$$E \coloneqq \CCC_{n-1}(k) = \{(\pi , \mathfrak{c}) | \pi \in \Av_{k}(\monodown_n), \, \mathfrak{c} \text{ is an inherited $(n-1)$-colouring of } \pi\}\, ,$$
	where the edge $(\pi , \mathfrak{c})$ connects $ v_1 \to v_2 $ with $v_1 = \be_{k-1}(\pi , \mathfrak{c})$ and $v_2 = \en_{k-1}(\pi , \mathfrak{c})$.
\end{defin}

In \cref{fig:ovk3n2} we present the coloured overlap graph corresponding to $k=3$ and $n=3$.

\begin{figure}
	\begin{minipage}[c]{0.33\textwidth}
		\centering
		\includegraphics[scale=0.6]{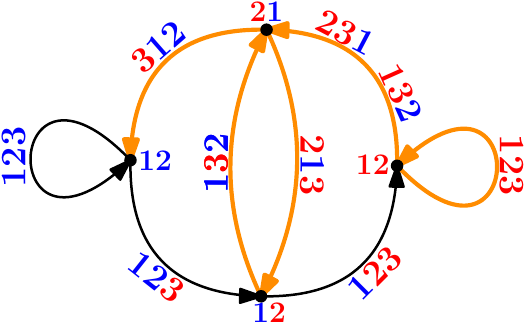}
	\end{minipage}\hfill
	\begin{minipage}[c]{0.55\textwidth}
		\captionsetup{width=\textwidth}
		\caption{The coloured overlap graph for $k=3$ and $n=3$, which also appears in the bottom part of \cref{fig:P_3_and_rest}.
			Note that in order to obtain a clearer picture we do not draw multiple edges, but we use multiple labels (for example the edge ${\color{red}12} \to {\color{red}2}{\color{blue}1}$ is labelled with the permutations ${\color{red}2}{\color{red}3}{\color{blue}1}$ and ${\color{red}1}{\color{red}3}{\color{blue}2}$ and should be thought of as two distinct edges labelled with ${\color{red}2}{\color{red}3}{\color{blue}1}$ and ${\color{red}1}{\color{red}3}{\color{blue}2}$ respectively). The role of the orange edges will be clarified later.
			\label{fig:ovk3n2}}
	\end{minipage}
\end{figure}

%\begin{figure}[htbp]
%	\centering
%	\includegraphics[scale=0.5]{ovk3n2_path.pdf}
%	\captionsetup{width=\textwidth}
%	\caption{The coloured overlap graph for $k=3$ and $n=3$, that also appears in the right-hand side of \cref{fig:P_3_and_rest}.
%		Note that in order to obtain a clearer picture we do not draw multiple edges, but we use multiple labels (for example the edge ${\color{red}12} \to {\color{red}2}{\color{blue}1}$ is labelled with the permutations ${\color{red}2}{\color{red}3}{\color{blue}1}$ and ${\color{red}1}{\color{red}3}{\color{blue}2}$ and should be thought of as two distinct edges labelled with ${\color{red}2}{\color{red}3}{\color{blue}1}$ and ${\color{red}1}{\color{red}3}{\color{blue}2}$ respectively). The role of the orange edges will be clarified later.
%	 \label{fig:ovk3n2}}
%\end{figure}

\begin{lem}\label{lemma:well_defn}
	The coloured overlap graph is well-defined, i.e.\ that for any edge $(\pi , \mathfrak{c})\in \CCC_{n-1}(k)$, then both $\be_{k-1}(\pi , \mathfrak{c})\in  \CCC_{n-1}(k-1)$ and $\en_{k-1}(\pi , \mathfrak{c})\in  \CCC_{n-1}(k-1)$.
\end{lem}

The following simple result is a key step for the proof of the lemma above.
\begin{obs}\label{obs:Correctness} For all permutations $\sigma\in\Av(\monodown_n)$ and all $j\leq |\sigma|$, we have that
	$$\be_j(\mathbb{S}(\sigma)) = \mathbb{S}(\be_j(\sigma))\, .$$
\end{obs}

\begin{proof}[Proof of \cref{lemma:well_defn}]
	 We can equivalently show that given an inherited $(n-1)$-colouring $(\pi , \mathfrak{c})$ of size $k$, both $\be_{k-1}(\pi , \mathfrak{c})$ and $\en_{k-1}(\pi , \mathfrak{c})$ are inherited $(n-1)$-colourings of size $k-1$.
	 
	Recall that we say an $(n-1)$-colouring $\mathfrak{c}$ of a permutation $\sigma\in\Av(\monodown_n)$ of size $k$ is \emph{inherited} if there is some permutation $\sigma\in\Av(\monodown_n)$ of size $\ell \geq k$ such that $\en_k ( \mathbb{S}(\sigma) )=(\pi, \mathfrak{c})$.
%	Therefore we can find $\sigma\in\Av(\monodown_n)$ such that $\en_k (\mathbb{S}(\sigma)) = (\pi , \mathfrak{c})$. 
	Then we have that $\en_{k-1} (\mathbb{S}(\sigma)) = \en_{k-1}(\pi , \mathfrak{c})$, and therefore $\en_{k-1}(\pi , \mathfrak{c})\in  \CCC_{n-1}(k-1)$. On the other hand, from \cref{obs:Correctness} we have that 
	\begin{multline}
	\label{eq:prop117trick}
	\be_{k-1}(\pi , \mathfrak{c}) =\be_{k-1}( \en_{k} ( \mathbb{S}(\sigma))) =\\
	 \en_{k-1} (  \be_{|\sigma| - 1} (\mathbb{S}(\sigma))) \stackrel{\ref{obs:Correctness}}{=} \en_{k-1} ( \mathbb{S}( \be_{|\sigma| - 1} (\sigma))) \, ,
	\end{multline}
	and so $\be_{k-1}(\pi , \mathfrak{c})\in  \CCC_{n-1}(k-1)$.
\end{proof}

We can now give a more precise formulation of \cref{thm:feasregmonotones*}. We recall that we denote by $\delta$ the Kronecker delta function.

\begin{thm}
	\label{thm:feasregmonotones}
	Let $\Pi$ be the projection map
	\begin{equation}
		\Pi : \mathbb{R}^{\mathcal C_{n-1}(k)} \to \mathbb{R}^{\Av_k(\monodown_n)}\label{eq:feasregmonotones}
	\end{equation}
	that sends the basis elements $(\delta_{(\pi, \mathfrak{c})}(x))_{x\in \mathcal C_{n-1}(k)}$ to $(\delta_{\pi}(x))_{x\in \Av_k(\monodown_n)}$, i.e.\ the map that ``forgets'' colourings.

	In this way, the feasible region $P^{\monodown_n}_k$ is the $\Pi$-projection of the cycle polytope of the overlap graph $\OvMon[k, \monodown_n]$. 
	That is,
	$$P^{\monodown_n}_k = \Pi ( P (\OvMon[k, \monodown_n]) )  \, .$$
\end{thm}

\subsection{The feasible region is the projection of the cycle polytope of the coloured overlap graph}

To prove \cref{thm:feasregmonotones}, we start by recalling that $P^{\monodown_n}_k $ is a convex set, as established in \cref{prop:convexity_312}.
Thus, in order to prove that $P^{\monodown_n}_k \supseteq \Pi ( P (\OvMon[k, \monodown_n]) )$ it is enough to show that for any vertex $\vec{v} \in P(\OvMon[k, \monodown_n])$ -- these vertices are given by the simple cycles of $\OvMon[k, \monodown_n]$ -- its projection $\Pi(\vec{v}) $ is in the feasible region. To this end, we construct a walk map $\Wpath[k, \monodown_n]$ (see \cref{defn:path_fnt} below) that transforms a permutation $\sigma\in\Av(\monodown_n)$ into a walk on the graph $\OvMon[k, \monodown_n]$. Secondly, in order to prove the other inclusion, we see via a factorization theorem that any point in the feasible region results from a sequence of walks in $\OvMon[k, \monodown_n]$ that can be asymptotically decomposed into simple cycles; so the feasible region must be in the convex hull of the vectors given by simple cycles.

\begin{defin}[The coloured walk function]\label{defn:path_fnt}
Let $\sigma $ be a permutation in $\Av_m(\monodown_n)$.
The walk $\Wpath[k, \monodown_n] (\sigma)$ is the walk of size $m-k+1$ on $\OvMon[k,\monodown_n]$  given by:
$$  \left(\pat_{\{1, \dots , k\}}(\mathbb{S}(\sigma )), \dots , \pat_{\{m-k+1, \dots , m\}}(\mathbb{S}(\sigma )) \right) \, , $$
where we recall that $\mathbb{S}(\sigma )=(\sigma,\mathbb{C}(\sigma ))$, with $\mathbb{C}(\sigma )$ the RITMO colouring of $\sigma$ presented in \cref{defn:colourings_of_permutations}.
\end{defin}

\begin{rem}
Given a permutation $\sigma $ that avoids $\monodown_n$, each of the restrictions 
$$ \pat_{\{\ell-k+1, \dots , \ell\}}(\mathbb{S}(\sigma )),\quad \text{for all}\quad \ell\in[k,m],$$
is an inherited $(n-1)$-colouring.
The fact that these are $(n-1)$-colourings follows because $\sigma $ avoids $\monodown_n $, and the fact that these are inherited colourings follows from \cref{obs:Correctness} after computations similar to \cref{eq:prop117trick}.
\end{rem}

\begin{exmp}
We present the walk $\Wpath[k, \monodown_n] (\sigma)$ corresponding to the permutation $\sigma=1243756$, for $k = 3$ and $n = 3$.
The RITMO colouring of $\sigma$ is ${\color{red} 1}{\color{red} 2}{\color{red} 4}{\color{blue} 3}{\color{red} 7}{\color{blue} 5}{\color{blue} 6}$, and the corresponding walk is
$({\color{red} 1}{\color{red} 2}{\color{red} 3}, {\color{red} 1}{\color{red} 3}{\color{blue} 2}, {\color{red} 2}{\color{blue} 1}{\color{red} 3}, {\color{blue} 1}{\color{red} 3}{\color{blue} 2}, {\color{red} 3}{\color{blue} 12})\, . $
We can see in \cref{fig:ovk3n2} this walk highlighted in orange on the coloured overlap graph $\OvMon[3, 321]$.
\end{exmp}

%\begin{figure}[htbp]
%\centering
%\includegraphics[scale=0.5]{ovk3n2_path.pdf}
%\caption{\label{fig:walk_example}The walk $\Wpath[3, 321] (1243756)$ in the coloured overlap graph $\OvMon[3, 321]$ is highlighted in orange.}
%\end{figure}

The following preliminary lemma is fundamental for the proof of \cref{thm:feasregmonotones}.

\begin{lm}\label{lemma:pathpermmono}
There exists a constant $C= C(k, n)$ such that, for any walk $w = (e_1, \dots , e_j) $ in $\OvMon[k, \monodown_n]$ there exists a walk $w' $ in $\OvMon[k, \monodown_n]$ of length $|w'|\leq C$ and a permutation $\sigma $ of size $j+k-1+|w'|$ that satisfies $\Wpath[k, \monodown_n](\sigma ) = w' \bullet w$.
\end{lm}

\begin{rem}
	Note that, heuristically speaking, \cref{lemma:pathpermmono} states that the map $\Wpath[k, \monodown_n]$ is ``almost" surjective. This gives an analogue of the result stated in \cref{lemma:pathperm312} for $\monodown_n$-avoiding permutations instead of 312-avoiding permutations.
\end{rem}

In the same spirit of the proof of \cref{lemma:lem_for_pathperm312}, in order to prove \cref{lemma:pathpermmono} we need the following result (whose proof is postponed to \cref{sect:main_lem_proof}). Recall the definition of the set $C_{G}(e)$ of continuations of an edge $e$ in a graph $G$, i.e.\ the set of edges $e' \in E(G)$ such that $\st(e') = \ar(e) $.

\begin{lem}\label{lemma:pathpermmono_part2}
	Let $\sigma$ be a permutation in $\Av(\monodown_n)$ such that $\en_k(\mathbb{S}(\sigma))=(\pi , \mathfrak{c})$ for some $(\pi , \mathfrak{c})\in E(\OvMon[k, \monodown_n])$. Assume that $\mathbb{C}(\sigma)$ is a rainbow $(n-1)$-colouring. Let also $(\pi' , \mathfrak{c}')\in C_{\OvMon[k, \monodown_n]}(\pi, \mathfrak{c})$.
	Then there exists $\iota\in [|\sigma| + 1]$ such that $\en_k(\mathbb{S}(\sigma^{*\iota }))  =(\pi' , \mathfrak{c}')$.
\end{lem}

\begin{proof}[Proof of \cref{lemma:pathpermmono}]
	We start by defining the desired constant $C=C(k,n)$.
	Recall that the edges of the coloured overlap graph $\OvMon[k, \monodown_n]$ are inherited colourings of permutations.
	Therefore, for each edge $e = (\pi, \mathfrak{c}) \in E(\OvMon[k, \monodown_n])$ we can choose $ \sigma_e$, one among the smallest $\monodown_n$-avoiding permutations  such that $(\pi, \mathfrak{c}) = \en_k(\mathbb{S}(\sigma_e))$.
	Define $C(k,n) \coloneqq \max_{e\in E(\OvMon[ k, \monodown_n])} |\sigma_e| + n - k - 1$.
	We claim that this is the desired constant.	
	
	\medskip
	
	We will prove a stronger version of the lemma, by constructing a permutation $\sigma\in \Av(\monodown_n)$ such that $\mathbb{C}(\sigma)$ is a rainbow $(n-1)$-colouring and $\Wpath[k, \monodown_n](\sigma ) = w' \bullet w$ for some walk $w'$ bounded as above.
	This will be proven by induction on the length of the walk $j = |w |$.

	\medskip
	
	We first consider the case $j=1$. 
	In this case, the walk $w  = (e_1)$ has a unique edge, and we can select $\sigma = (n-1) \cdots 1 \oplus \sigma_{e_1}$.
	In this way, it is clear that $\mathbb{C}(\sigma)$ is a rainbow $(n-1)$-colouring, because $\sigma$ has a monotone decreasing subsequence of size $n-1$, while it is clearly $\monodown_n$-avoiding.
	Furthermore, because $\en_k(\mathbb{S}(\sigma) ) = \en_k(\mathbb{S}(\sigma_{e_1}) ) = e_1$, we have that $\Wpath[k, \monodown_n](\sigma )= w ' \bullet e_1$ for some path $w '$ such that $|w ' \bullet e_1|=|w '|+1= 
	|\sigma| - k + 1 = |\sigma_{e_1}| + n - k$.
	Therefore we have that $|w '|= 
	|\sigma_{e_1}| + n - 1 - k \leq C$, concluding the base case.
	
	\medskip
	
	We now consider the case $j\geq  2$. Take a walk $w = (e_1, \dots , e_j)$ in $\OvMon[k, \monodown_n]$, and consider (by induction hypothesis) the permutation $\sigma$ such that $\Wpath[k, \monodown_n](\sigma ) = w ' \bullet  (e_1, \dots , e_{j-1})$ for some walk $w '$ of size at most $C$ and such that $\mathbb{C}(\sigma)$ is a rainbow $(n-1)$-colouring.
	
	From \cref{lemma:pathpermmono_part2}, we can find a value $\iota\in [|\sigma| + 1]$ such that $\en_k(\mathbb{S}(\sigma^{*\iota }))  =e_j$.
	If so, the colouring $\mathbb{C}(\sigma^{*\iota}) $ is clearly a rainbow $(n-1)$-colouring (hence, $\sigma^{*\iota}\in \Av(\monodown_n)$).
	Furthermore, we have that $\Wpath[k, \monodown_n](\sigma^{*\iota}) = w ' \bullet  (e_1, \dots , e_{j-1}, e_j)$, concluding the induction step, as $|w '|\leq C$ by hypothesis.
\end{proof}

We recover here a proposition from \cite{borga2019feasible} that will be important in establishing \cref{thm:feasregmonotones}.

\begin{prop}[Vertices of the cycle polytope]
Let $G$ be a directed graph.
The set of vertices of $P(G)$ is precisely $\{\, \vec{e}_{\CCC}\,  |\,  \CCC \text{ is a simple cycle of } G\}$.
\end{prop}

We can now prove the main result of this section.

\begin{proof}[Proof of \cref{thm:feasregmonotones}]
Let $\sigma\in\Av(\monodown_n)$. Let us first establish a formula for $\pcoc_k(\sigma)$ with respect to the walk $\Wpath[k, \monodown_n](\sigma)$ defined in \cref{defn:path_fnt}. Given a permutation $\rho$ with a colouring $\mathfrak{c}$ we set $\per(\rho,\mathfrak{c})=\rho$.
Given a walk $w$ in $\OvMon[k, \monodown_n ]$ and a permutation $\pi$, we define $[\pi : w]$ as the number of edges $e$ in $w$ such that $\per(e) = \pi$.
Thus, it easily follows that
\begin{equation}
\label{eq:formula_feas_vector}
\pcoc_k(\sigma ) = \frac{1}{|\sigma | } \sum_{\pi \in \Av_k(\monodown_n)} [\pi : \Wpath[k, \monodown_n](\sigma) ] \vec{e}_{\pi}  \, .
\end{equation}
On the other hand, using \cite[Proposition 2.2]{borga2019feasible}, the vertices of $P(\OvMon[k, \monodown_n])$ are given by the simple cycles of the graph $\OvMon[k, \monodown_n]$.
Specifically, the vertices are given by the vectors $\vec{e}_{\CCC}\in \mathbb{R}^{\mathcal C_{n-1}(k)}$, for each simple cycle $\mathcal C$ of $\OvMon[k, \monodown_n]$, as follows: 
$$(\vec{e}_{\CCC})_{(\pi, \mathfrak{c})} = \frac{\mathbb{1}[(\pi, \mathfrak{c}) \in \CCC]}{|\CCC|}\, , $$
for each inherited coloured permutation $(\pi, \mathfrak{c})$.
In this way, we have that 
\begin{equation}
\label{eq:proj_rel}
	\Pi (\vec{e}_{\CCC} ) = \frac{1}{|\CCC|}\sum_{\pi\in\Av_k(\monodown_n)} [\pi : \CCC]\vec{e}_{\pi} \, .
\end{equation}  

Now let us start by proving the inclusion $\Pi ( P (\OvMon[k, \monodown_n]) ) \subseteq P^{\monodown_n}_k $.
Take a vertex of the polytope $P(\OvMon[k,\monodown_n ]) $, that is a vector $\vec{e}_{\CCC}$ for some simple cycle $\CCC$ of $\OvMon[k,\monodown_n ]$.
Because $\CCC$ is a cycle, we can define the walk $\CCC^{\bullet \ell}$ obtained by concatenating $\ell$ times the cycle $\CCC$. From \cref{lemma:pathpermmono}, there exists a walk $w'_\ell$ with $|w'_{\ell}| \leq C(k,n)$  and a $\monodown_n$-avoiding permutation $\sigma^{\ell}$ of size $|w'_\ell|+\ell|\CCC|+k-1$, such that $\Wpath[k, \monodown_n](\sigma^{\ell} ) = w'_{\ell} \bullet \CCC^{\bullet \ell}$.
The next step is to prove that 
$$\pcoc_k(\sigma^{\ell})\xrightarrow{\ell\to \infty} \Pi(\vec{e}_{\CCC}).$$
We have that 
\begin{align*}
\pcoc_k(\sigma^{\ell}) &\stackrel{\eqref{eq:formula_feas_vector}}{=} \frac{1}{|\sigma^{\ell} | } \sum_{\pi \in \Av_k(\monodown_n)} [\pi : \Wpath[k, \monodown_n](\sigma^{\ell}) ] \vec{e}_{\pi}\\
&=\frac{\ell}{|\sigma^{\ell} | } \left(\sum_{\pi \in \Av_k(\monodown_n)}  [\pi : \CCC ] \vec{e}_{\pi} \right) + \frac{1}{|\sigma^{\ell} | }\sum_{\pi \in \Av_k(\monodown_n)} [\pi : w'_{\ell} ] \vec{e}_{\pi} \\
	&\stackrel{\eqref{eq:proj_rel}}{=} \frac{\ell |\CCC|}{|\sigma^{\ell} | }\Pi (\vec{e}_{\CCC} ) + \frac{1}{|\sigma^{\ell} | }\vec{z}_\ell = \left( 1 -  \frac{ k - 1 + |w'_{\ell}| }{|\sigma^{\ell} | } \right)\Pi (\vec{e}_{\CCC} ) + \frac{1}{|\sigma^{\ell} | }\vec{z}_\ell\, ,
\end{align*}
where $\vec{z}_\ell = \sum_{\pi \in \Av_k(\monodown_n)} [\pi : w'_{\ell} ] \vec{e}_{\pi}$.
However, because $|w'_{\ell}| \leq C(k,n)$, we have that 

\begin{equation*}
	\frac{ k - 1 + |w'_{\ell}| }{|\sigma^{\ell} | } \xrightarrow{\ell\to \infty} 0\quad\text{and}\quad\frac{1}{|\sigma^{\ell} | }|| \vec{z}_\ell||_1 =  \frac{1}{|\sigma^{\ell} | }\sum_{\pi \in \Av_k(\monodown_n)} [\pi : w'_{\ell} ] = \frac{|w'_{\ell}| }{|\sigma^{\ell}|} \xrightarrow{\ell\to \infty} 0\,  . 
\end{equation*}
Therefore $\pcoc_k(\sigma^{\ell}) \to \Pi (\vec{e}_{\CCC} )$.
This, together with \cref{prop:convexity_312}, shows the desired inclusion.

\medskip

For the other inclusion, consider $\vec{v} \in P^{\monodown_n}_k$, so that there is a sequence of $\monodown_n$-avoiding permutations $\sigma^{\ell}$ such that $\pcoc_k(\sigma^{\ell})\xrightarrow{\ell\to \infty} \vec{v}$ and that $|\sigma^{\ell}| \xrightarrow{\ell\to \infty}  \infty$.
Fix $\varepsilon > 0$, and let $M $ be an integer such that $\ell\geq M $ implies $|| \pcoc_k(\sigma^{\ell}) - \vec{v} ||_2 < \frac{\varepsilon}{2}$ and $|\sigma^{\ell}|>  \frac{6 k!}{\varepsilon } $.
The set of edges of the walk $\Wpath[k, \monodown_n](\sigma^{\ell})$ can be split into $\CCC_1^{(\ell)} \uplus \dots \uplus \CCC_j^{(\ell)} \uplus \mathcal T^{(\ell)} $, where each $\CCC_i^{(\ell)} $ is a simple cycle of $\OvMon[k, \monodown_n]$ and $\mathcal T^{(\ell)} $ is a path that does not repeat vertices, so $|\mathcal T^{(\ell)}| < V(\OvMon[k, \monodown_n] ) \leq (k-1)!$ (for a precise explanation of this fact see \cite[Lemma 3.13]{borga2019feasible}).
Thus, we get
\begin{multline*}
\pcoc_k(\sigma^{\ell}) \stackrel{\eqref{eq:formula_feas_vector}}{=} \frac{1}{|\sigma^{\ell} | } \sum_{\pi \in \Av_k(\monodown_n)} [\pi : \Wpath[k, \monodown_n](\sigma^{\ell} ) ] \vec{e}_{\pi}\\
=\frac{1}{|\sigma^{\ell} | } \sum_{i=1}^j \sum_{\pi \in \Av_k(\monodown_n)} [\pi : \CCC_i^{(\ell)} ] \vec{e}_{\pi} + \frac{1}{|\sigma^{\ell} | } \sum_{\pi \in \Av_k(\monodown_n)} [\pi : \mathcal T^{(\ell)} ] \vec{e}_{\pi}  \\
   \stackrel{\eqref{eq:proj_rel}}{=}  \frac{|\sigma^{\ell} |  - |\mathcal T^{(\ell)}| - k + 1}{|\sigma^{\ell} |} \sum_{i=1}^j \frac{|\CCC_i^{(\ell)}|}{|\sigma^{\ell} |  - |\mathcal T^{(\ell)}| - k + 1 } \Pi(\vec{e}_{\CCC_i^{(\ell)}} ) + \frac{1}{|\sigma^{\ell} | } \sum_{\pi \in \Av_k(\monodown_n)} [\pi : \mathcal T^{(\ell)} ] \vec{e}_{\pi} \, .
\end{multline*}
Now we set $\vec{x}\coloneqq \sum_{i=1}^j \frac{|\CCC_i^{(\ell)}|}{|\sigma^{\ell} |  - |\mathcal T^{(\ell)}| - k + 1 } \Pi(\vec{e}_{\CCC_i^{(\ell)}} ) $ and $\vec{y}\coloneqq\frac{1}{|\sigma^{\ell} | } \sum_{\pi \in \Av_k(\monodown_n)} [\pi : \mathcal T^{(\ell)} ] \vec{e}_{\pi}$. 
Note that $\vec{x} \in \Pi (P(\OvMon[k, \monodown_n ] ) ) $; indeed it is a convex combination (since $\sum_{i=1}^j |\CCC_i^{(\ell)}| = |\sigma^{\ell} |  - |\mathcal T^{(\ell)}| - k + 1$) of vectors corresponding to simple cycles.
We simply get that
$$\pcoc_k(\sigma^{\ell}) =   \frac{|\sigma^{\ell} |  - |\mathcal T^{(\ell)}| - k + 1}{|\sigma^{\ell} |} \vec{x} + \vec{y}\, . $$
Thus, 
\begin{multline}
\label{eq:dist_approx}
\dist\left(\pcoc_k(\sigma^{\ell}) , \Pi \left(P(\OvMon[k, \monodown_n ] ) \right)\right) \leq  ||\pcoc_k(\sigma^{\ell}) - \vec{x}||_2\\
  \leq\frac{|\mathcal T^{(\ell)}| + k - 1}{|\sigma^{\ell} |}|| \vec{x}||_2 + ||\vec{y}||_2\, .
\end{multline}

Observe that $||\vec{y}||_2 \leq \frac{1}{|\sigma^{\ell} | }\sum_{\pi \in \Av_k(\monodown_n)} [\pi : \mathcal T^{(\ell)} ] = \frac{|\mathcal{T}^{(\ell)}|}{|\sigma^{\ell} | }\leq \frac{(k-1)!}{|\sigma^{\ell} | }$.
Also, because the coordinates of $ \vec{x}$ are non-negative and sum to one, we have that $|| \vec{x}||_2 \leq 1$ and so that $\frac{|T^{(\ell)}| + k - 1}{|\sigma^{\ell} |} ||\vec{x}||_2 \leq \frac{(k-1)! + k - 1}{|\sigma^{\ell} |}$.
Then, we can simplify \cref{eq:dist_approx} to
$$ \dist(\pcoc_k(\sigma^{\ell}) , \Pi (P(\OvMon[k, \monodown_n ] ) )) \leq  \frac{(k-1)! + k  - 1 + (k-1)!}{|\sigma^{\ell} |}\leq \frac{3 k!}{|\sigma^{\ell} |}\, ,$$
so that for $\ell \geq M$ we have that $ \dist(\pcoc_k(\sigma^{\ell}) , \Pi (P(\OvMon[k, \monodown_n ] ) ))  < \frac{1}{2}\varepsilon$.
As a consequence, for $\ell\geq M$,
\begin{multline}
		\dist(\vec{v}, \Pi (P(\OvMon[k, \monodown_n ] ) ))\\
		\leq || \vec{v} - \pcoc_k(\sigma^{\ell}) ||_2 + \dist (\pcoc_k(\sigma^{\ell}), \Pi (P(\OvMon[k, \monodown_n ] ) )) <  \varepsilon \, . 
\end{multline}
Noting that $\Pi (P(\OvMon[k, \monodown_n ] ) ) $ is a closed set, since $\varepsilon $ is generic, we obtain that $\vec{v}$ is in the polytope $\Pi (P(\OvMon[k, \monodown_n ] ) )$, concluding the proof of the theorem.
\end{proof}

It just remains to prove \cref{lemma:pathpermmono_part2}. This is the goal of the next two sections.

\subsection{Preliminary results: basic properties of RITMO colourings and their relations with active sites}

We begin by stating (without proof) some basic properties of the RITMO colouring. We suggest to compare the following lemma with \cref{fig:RITMO_props}. We remark that Properties 2 and 3 in the following lemma arise as particular cases of a general result explained after the statement.

\begin{lm}\label{lm:RITMO_Properties}
Let $\sigma $ be a permutation, and consider $\mathbb{C}(\sigma) $ its RITMO colouring.

\begin{enumerate}
\item If $i < j \in [|\sigma |]$ such that $\sigma(i) > \sigma(j)$, then $\mathbb{C}(\sigma)(i ) < \mathbb{C}(\sigma)(j)$.

\item If $i < j \in [|\sigma |]$ such that $\sigma(i) < \sigma(j)$ and $\mathbb{C}(\sigma)(i ) < \mathbb{C}(\sigma)(j )$, then there exists $k$ such that $i< k < j$, $\sigma(k)>\sigma(j)$ and $\mathbb{C}(\sigma)(k) = \mathbb{C}(\sigma)(i)$.

\item If $i < j \in [|\sigma |]$ such that $\sigma(i) < \sigma(j)$ and $\mathbb{C}(\sigma)(i ) < \mathbb{C}(\sigma)(j )$, then there exists $h$ such that  $i< h < j$, $\sigma(h)>\sigma(j)$ and $\mathbb{C}(\sigma)(h) = \mathbb{C}(\sigma)(j)-1$.
\end{enumerate}
\end{lm}

As mentioned before, we explain that Properties 2 and 3 are particular cases of the same general result: consider $i < j \in [|\sigma |]$ with $\sigma(i) < \sigma(j)$. Let $c=\mathbb{C}(\sigma)(i)$ and $d=\mathbb{C}(\sigma)(j)$ and assume that $c < d$. Then there are indices $i<k_c < k_{c+1} < \dots < k_{d-1} < j $ such that $\mathbb{C}(\sigma)(k_s) = s$ for all $s \in [c, d-1]$, and $\sigma(j)<\sigma(k_{d-1})  < \dots < \sigma(k_c)$.
We opt to single out Properties 2 and 3 because these will be enough for our applications.

\begin{figure}[htbp]
	\begin{minipage}[c]{0.33\textwidth}
		\centering
		\includegraphics[scale=0.6]{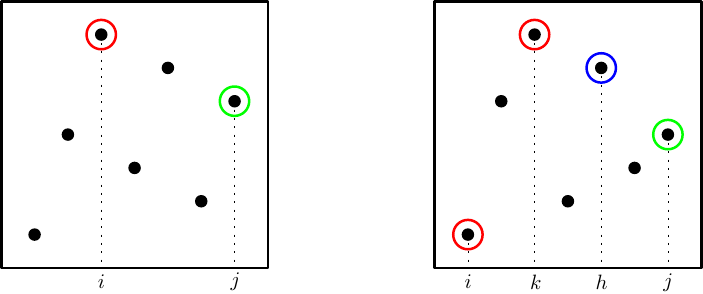}
	\end{minipage}\hfill
	\begin{minipage}[c]{0.35\textwidth}
		\captionsetup{width=\textwidth}
		\caption{\label{fig:RITMO_props} A schema for \cref{lm:RITMO_Properties}. The left-hand side illustrates Property 1 and the right-hand side illustrates Properties 2 and 3.}
	\end{minipage}
\end{figure}

We now introduce a key definition.

\begin{defin}\label{defn_col_active_sites}
	Given a coloured permutation $(\pi, \mathfrak{c})$, and a pair $(y, f)$ with $y \in [|\pi|+1]$, $f\geq 1$, we define the coloured permutation $(\pi, \mathfrak{c})^{*(y, f)}$ to be the permutation $\pi^{*y}$ (see \cref{defin:append_permutation}) together with the colouring $\mathfrak{c}^{*f}$.
	The latter is defined as a colouring $\mathfrak{c}^{*f}:[|\pi|+1] \to \Z_{\geq 1}$ such that $\mathfrak{c}^{*f}(i)=\mathfrak{c}(i)$ for all $i\in[|\pi|]$ and $\mathfrak{c}^{*f}(|\pi|+1) = f$.
	
	Let $(\pi, \mathfrak{c})$ be an inherited $(n-1)$-coloured permutation.
	An \emph{active site} is a pair $(y, f)$ with $y\in [|\pi|+1]$ and $f \in  [n-1]$, such that $(\pi, \mathfrak{c})^{*(y, f)}$ is an inherited $(n-1)$-coloured permutation.
\end{defin}

We present the following analogue of \cref{lm:RITMO_Properties}.

\begin{lm}\label{lm:ActiveSites_Properties}
Let $(y, f)$ be an active site of an inherited coloured permutation $(\pi, \mathfrak{c})$, and consider some index $i \in[|\pi|]$.
Then

\begin{enumerate}
\item if $\mathfrak{c} ( i ) \geq f$, then $y > \pi(i)$;

\item if $\pi ( i ) < y$ and $\mathfrak{c} ( i ) < f$, then there exists $k>i$ such that $\pi(k)\geq y$ and  $\mathfrak{c}(k) = \mathfrak{c}(i)$.

\item if $\pi ( i ) < y$ and $\mathfrak{c} ( i ) < f$, then there exists $h>i$ such that $\pi(h)\geq y$ and  $\mathfrak{c}(h) = f-1$.
\end{enumerate}

\end{lm}

\begin{proof}
Let $\sigma $ be a permutation such that $\en_{|\pi|+1}(\mathbb{S}(\sigma)) =(\pi, \mathfrak{c})^{*(y, f)}$, which exists because $(y, f)$ is an active site of $(\pi, \mathfrak{c})$.
The lemma is an immediate consequence of \cref{lm:RITMO_Properties}, applied to the RITMO colouring $\mathbb{C}(\sigma)$, and for $j = | \sigma |$ (so that $\mathbb{C}(\sigma)(j)=f$ and $\sigma(j)=y$). 
\end{proof}

We now observe a correspondence between edges of $\OvMon[k, \monodown_n]$ and active sites of some coloured permutations.

\begin{obs}\label{obs:active_sites_edge}
	Fix an inherited coloured permutation $(\pi_1, \mathfrak{c}_1)$ of size $k-1$.
	Then there exists a bijection between the set of edges $e\in \OvMon[k, \monodown_n]$ with $\st(e) = (\pi_1, \mathfrak{c}_1)$ and the set of active sites $(y, f)$ of $(\pi_1, \mathfrak{c}_1)$.
	Specifically, this correspondence between edges and active sites is given by the following two maps, which can be easily seen to be inverses of each other:
	$$e = (\pi, \mathfrak{c}) \mapsto (\pi(k), \mathfrak{c}(k)), \quad (y, f) \mapsto (\pi_1, \mathfrak{c}_1)^{*(y, f)}\, . $$
\end{obs}

Fix now an inherited coloured permutation  $(\pi, \mathfrak{c})$. By definition, there exists some $\sigma_0$ that satisfies $\en_{|\pi|}(\mathbb{S}(\sigma_0)) =  (\pi, \mathfrak{c})$.
The goal of the next section is to show that, with some mild restrictions on the chosen permutation $\sigma_0 $, if $(y, f)$ is an active site of $(\pi, \mathfrak{c}) $ then there exists an index $i\in[|\sigma_0|+1]$ such that 
$$\en_{|\pi|+1}(\mathbb{S}(\sigma_0^{*i})) =  (\pi, \mathfrak{c})^{*(y, f)} \, .$$
We already know that there exists a permutation $\sigma_1$ such that $\en_{|\pi|+1}(\mathbb{S}(\sigma_1)) =  (\pi, \mathfrak{c})^{*(y, f)}$; here we are interested in finding out if $\sigma_1$ can arise as an extension of $\sigma_0$.

\medskip

We introduce two definitions and give some of their simple properties.

\begin{defin}\label{defin:widetilde}
Let $\pi$ and $\sigma$ be two permutations such that $\pi = \en_{k-1}(\sigma)$.
For a point at height $\ell\in[|\pi|]$ in the diagram of $\pi$, we define $\widetilde{\ell}$ to be the height of the corresponding point in the diagram of $\sigma$.
Algebraically we have that $\widetilde{\ell}= \sigma(|\sigma | - |\pi | + \pi^{-1}(\ell))$. 
We use the convention that $\widetilde{|\pi|+1 } = |\sigma| +1$ and $\widetilde{0} = 0$.
\end{defin}

See \cref{fig:example_for_defn_tilde} for an example. We have the following simple result.

\begin{figure}[htbp]
	\begin{minipage}[c]{0.33\textwidth}
		\centering
		\includegraphics[scale=0.53]{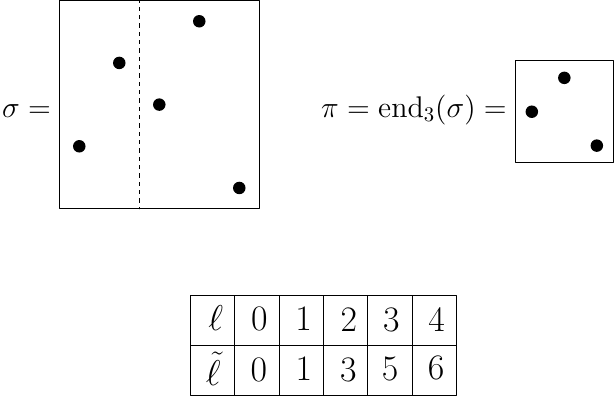}
	\end{minipage}\hfill
	\begin{minipage}[c]{0.43\textwidth}
		\captionsetup{width=\textwidth}
		\caption{\label{fig:example_for_defn_tilde}A schema illustrating \cref{defin:widetilde}. On top-left the permutation $\sigma = 24351$, on top-right the pattern $\pi = 231$ induced by the last three indices of $\sigma$, and on the bottom the quantities $\tilde{\ell}$. }
	\end{minipage}
\end{figure}

\begin{lm}\label{lm:indexentry}
Let $\sigma, \pi $ be permutations such that $\pi = \en_{k-1}(\sigma)$. 
Let $y \in  [|\pi |+1]$ and $\iota \in [|\sigma |+1]$.
Then we have that 
$$\en_k(\sigma^{* \iota} )= \pi^{* y} \iff \widetilde{y-1} < \iota \leq \widetilde{y}\, . $$

\end{lm}

\begin{defin}\label{defn:z_sigma_funct}
	
Fix a permutation $\sigma$ and a colour $f\in\{1, 2, \dots\}$. If there exists a maximal index $p$ of $\sigma$ such that $\mathbb{C}(\sigma)(p) = f$ we set $z_{\sigma}(f) \coloneqq \sigma(p)+1$. Otherwise, if such a $p$ does not exist, then $z_{\sigma}(f) \coloneqq 1$.	
We use the convention that $z_{\sigma} (0) = |\sigma|+2$.
\end{defin}

See \cref{fig:example_for_defn} for an example. We have the following simple result.

\begin{figure}[htbp]
	\begin{minipage}[c]{0.55\textwidth}
		\centering
		\includegraphics[scale=0.4]{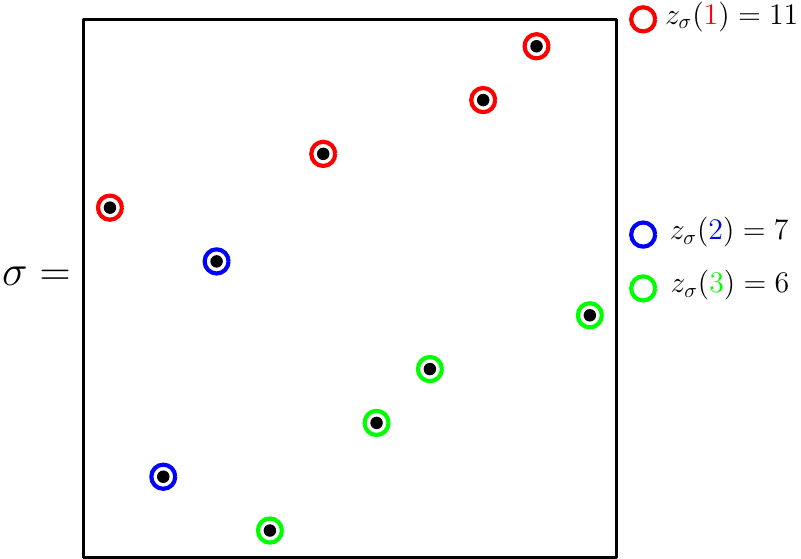}
	\end{minipage}\hfill
	\begin{minipage}[c]{0.43\textwidth}
		\captionsetup{width=\textwidth}
		\caption{\label{fig:example_for_defn}  A permutation $\sigma\in\Av(4321)$ coloured with its RITMO colouring. The quantities $z_{\sigma}(1),z_{\sigma}(2),z_{\sigma}(3)$, defined in \cref{defn:z_sigma_funct}, are highlighted on the right of the diagram of the permutation $\sigma$.}
	\end{minipage}
\end{figure}

\begin{lm}\label{lm:indexcolour}
Let $\sigma $ be a permutation, $f \in \mathbb{Z}_{\geq 1}$ a colour and $\iota \in [|\sigma|+1]$.
Then we have that 
$$\mathbb{C}(\sigma^{*\iota})(|\sigma|+1) = f \iff z_{\sigma}(f) \leq \iota < z_{\sigma}(f-1)\, .$$
\end{lm}

\subsection{The proof of the main lemma}
We can now prove \cref{lemma:pathpermmono_part2}. 
We will do this as follows: in order to construct a suitable extension of the permutation $\sigma$, we will find a suitable index $\iota $ so that $\sigma^{*\iota}$ has the desired coloured pattern at the end.
According to \cref{lm:indexentry}, fixing the pattern at the end of $\sigma^{*\iota}$ determines an interval of admissible values for $\iota $, and according to \cref{lm:indexcolour}, fixing the colour of the last entry determines a second interval of admissible values for $\iota$.
The key step of the proof is to show that these two intervals have non-trivial intersection.

\label{sect:main_lem_proof}

\begin{proof}[Proof of \cref{lemma:pathpermmono_part2}]
Observe that $\be_{k-1}(\pi' , \mathfrak{c}') = \en_{k-1}(\pi, \mathfrak{c})= \en_{k-1}(\mathbb{S}(\sigma))$.
Let $(\rho, \mathfrak{d})= \en_{k-1}(\mathbb{S}(\sigma))$ be this common coloured permutation.
For an entry of height $\ell\in[|\rho|]$ in the diagram of $\rho$, we recall that $\widetilde{\ell}\in[|\sigma|]$ denotes the height of the corresponding entry in the diagram of $\sigma$, as in \cref{defin:widetilde}.
Let $(y, f)$ be the active site of $(\rho, \mathfrak{d})$ corresponding to the edge $(\pi' , \mathfrak{c}')$, so that $f \in [n-1]$ and $y \in [|\rho|+1]$ (see \cref{obs:active_sites_edge}).

From \cref{lm:indexentry}, we have that $\en_k(\sigma^{*\iota})= \pi'$ if and only if 
\begin{equation}\label{eq:interval_1}
 \widetilde{y-1} < \iota \leq \widetilde{y}\, . 
\end{equation}
From \cref{lm:indexcolour}, we have that $\mathbb{C}(\sigma^{*\iota})(|\sigma|+1)= f $ if and only if 
\begin{equation}\label{eq:interval_2}
z_{\sigma}(f) \leq \iota < z_{\sigma}(f-1)\, . 
\end{equation}
This gives us two intervals that are, by \cref{defin:widetilde} and \cref{defn:z_sigma_funct}, non-empty.
Our goal is to show that these intervals have a non-trivial intersection, concluding that the desired index $\iota $ exists.

\medskip

\textbf{Claim.} $z_{\sigma}(f) \leq \widetilde{y}$.

\medskip 

Assume by sake of contradiction that $z_{\sigma}(f) > \widetilde{y}$.
If $y = |\rho|+1$, then $\widetilde{y} = |\sigma|+1$ by convention. This gives a contradiction because $f\geq 1$ and so $z_{\sigma}(f)\leq |\sigma|+1$. 
Thus $y < |\rho|+1$. 
Let $p\in [|\sigma |]$ be the maximal index such that $\mathbb{C}(\sigma)(p) = f$.
We know that such a $p$ exists, because $\mathbb{C}(\sigma )$ is a rainbow $(n-1)$-colouring.
By maximality of $p$, it follows that $\sigma(p)+1 = z_{\sigma}(f)$ (see \cref{defn:z_sigma_funct}).
We now split the proof into two cases: when $p$ is included in the last $|\rho|$ indices of $\sigma$ and when it is not.

\begin{itemize}
	\item \textbf{Assume that $p> |\sigma|- |\rho|$}. 
	Let $q = p - (|\sigma|- |\rho| )>0$. Because $\en_{k-1}(\mathbb{S}(\sigma)) = (\rho, \mathfrak{d})$, we have that $f = \mathbb{C}(\sigma )(p) = \mathfrak{d}( q )$.
	Since we know that $\sigma(p )+1 = z_{\sigma}(f) > \widetilde{y}$, we have that $\rho(q)+1 > y $.
	This contradicts Property 1 of \cref{lm:ActiveSites_Properties}, as the active site $(y, f)$ satisfies both $\mathfrak{d}( q )\geq f$ and $\rho(q) \geq y$.
	\item \textbf{Assume that $p\leq |\sigma|- |\rho|$}. Then $\sigma(p) \neq \widetilde{y}$, so from $\sigma(p)+1= z_{\sigma}(f) > \tilde{y}$ we have that $\sigma(p) >  \widetilde{y}$.
	Using Property 1 of \cref{lm:RITMO_Properties} with $ i = p $ and $j = \sigma^{-1}(\tilde{y})$, we have that $f = \mathbb{C}(\sigma)(p) < \mathbb{C}(\sigma)( \sigma^{-1}(\widetilde{y}) )$.
	So $\mathfrak{d}( \rho^{-1}(y) ) = \mathbb{C}(\sigma)( \sigma^{-1}(\widetilde{y}) ) > f$.
	But this contradicts again Property 1 of \cref{lm:ActiveSites_Properties} for $i = \rho^{-1}(y)$, as the active site $(y, f)$ satisfies both $\mathfrak{d}( \rho^{-1}(y) )> f$ and $y \leq \rho(\rho^{-1}(y))$.	
\end{itemize}
Therefore, in both cases we have a contradiction.

\medskip

\textbf{Claim.} $z_{\sigma}(f-1) > \widetilde{y-1}+1$.

\medskip

 Assume by contradiction  that $z_{\sigma}(f-1) \leq \widetilde{y-1} + 1$.
If $f=1$, then recall that we use the convention that $z_{\sigma }(0) = |\sigma | + 2$, so we have 
$\widetilde{y-1} \geq  |\sigma | + 1$.
But $y \leq |\rho|+1$ so $\widetilde{y-1} \leq |\sigma|$, a contradiction.
Thus $f > 1$.
Let $p$ be the maximal index in $[|\sigma |]$ such that $\mathbb{C}(\sigma)(p) = f-1$.
We know that such a $p$ exists, because $\mathbb{C}(\sigma )$ is a rainbow $(n-1)$-colouring.
By construction, $\sigma(p)+1 = z_{\sigma}(f-1) \leq \widetilde{y-1}+1$ (see \cref{defn:z_sigma_funct}), so $\sigma(p)\leq \widetilde{y-1}$. 
As above, we now split the proof into two cases: when $p$ is included in the last $|\rho|$ indices of $\sigma$ and when it is not.

\begin{itemize}
\item \textbf{Assume that $p> |\sigma|- |\rho|$}. 
Let $q = p - (|\sigma|- |\rho| )>0$. 
Because $\en_{k-1}(\mathbb{S}(\sigma)) = (\rho, \mathfrak{d})$, we have that $f - 1 = \mathbb{C}(\sigma )(p) = \mathfrak{d}( q )$.
Since we know that $\sigma(p) \leq \widetilde{y-1}$, we have that $\rho(q) \leq y-1$.
Thus, by Property 2 of \cref{lm:ActiveSites_Properties}, there exists some $k>q$ such that  $\mathfrak{d}(k) = \mathfrak{d}( q ) = f - 1$.
The existence of such $k$ contradicts the maximality of $p$, as we get that $k + (|\sigma|- |\rho| ) > p $ has $\mathbb{C}(\sigma) (k + (|\sigma|- |\rho| )) = \mathfrak{d}(k) = f-1$.
\item \textbf{Assume that $p\leq |\sigma|- |\rho|$}. 
Let $r=\sigma^{-1}(\widetilde{y - 1})$.
Then $r > |\sigma|- |\rho| \geq p$ and so $p\neq r$.
It follows that $\sigma(p) = z_{\sigma}(f-1)-1 < \widetilde{y-1}$.

We now claim that $\mathbb{C}(\sigma )(r) < f - 1$.
Indeed, if $\mathbb{C}(\sigma )(r) = f-1$, because $p < r$ we have immediately a contradiction with the maximality of $p$.
Moreover, if $\mathbb{C}(\sigma )(r) > f-1$, Property 2 of \cref{lm:RITMO_Properties} guarantees that there is some $k > p$ such that $\sigma(k) > \sigma(r)$ and $\mathbb{C}(\sigma )(k)=f-1$.
Again, we have a contradiction with the maximality of $p$.

Now let $q=r-(|\sigma| - |\rho|)$, and observe that $\mathfrak{d}(q) = \mathbb{C}(\sigma)(r) < f - 1$.
On the other hand, because $r=\sigma^{-1}(\widetilde{y - 1})$, we have $\rho(q) = y-1$.
Because $(y, f)$ is an active site of $(\rho, \mathfrak{d})$, Property 3 of \cref{lm:ActiveSites_Properties} guarantees that there is some index $k> q$ of $\rho$ such that $\mathfrak{d}(k)= f-1$.
But this contradicts again the maximality of $p$, as we would have that $\mathbb{C}(\sigma)(k+|\sigma| - |\rho|) = f-1$ while $k+|\sigma| - |\rho| > q +|\sigma| - |\rho| = r > p$.
\end{itemize}
Therefore, in both cases we have a contradiction.

\medskip

Using the two claims above, we can conclude that the intervals in \cref{eq:interval_1,eq:interval_2} have a non-trivial intersection, and therefore the envisaged index $\iota$ we were looking for exists.
Consequently, we can construct the desired permutation $\sigma^{*\iota }$.
\end{proof}

\section*{Acknowledgements}
The authors are very grateful to Valentin F\'eray and Mathilde Bouvel for some precious discussions during the preparation of this paper.

\bibliography{mybib}
\bibliographystyle{alpha}

\end{document}